\documentclass[reqno]{amsart}
\usepackage{amsmath,amssymb,latexsym, amsfonts}

\usepackage{verbatim}
\usepackage{times}
\usepackage{hyperref}

\numberwithin{equation}{section}

\theoremstyle{plain}

\newtheorem{theorem}{Theorem}[section]
\newtheorem{thm}[theorem]{Theorem}

\newtheorem{prop}[theorem]{Proposition}
\newtheorem{lemma}[theorem]{Lemma}
\newtheorem{lem}[theorem]{Lemma}

\theoremstyle{definition}
\newtheorem{definition}[theorem]{Definition}
\newtheorem{dfn}[theorem]{Definition}

\newcommand{\boehmcomplex}{B}
\newcommand{\hatC}{C}

\theoremstyle{remark}

\newtheorem{rem}{Remark}

\newcommand{\ahha}{{\scriptscriptstyle{A}}}

\newcommand{\emme}{{\scriptscriptstyle{M}}}

\newcommand{\gd}{\delta} 
\newcommand{\gD}{\Delta}
\newcommand{\Deltaell}{\Delta} 

\newcommand{\gl}{\lambda} 
\newcommand{\gL}{\Lambda}

\newcommand{\gs}{\sigma}

\newcommand{\eps}{\epsilon}

\newcommand{\Tor}{{\rm Tor}}
\newcommand{\Ext}{{\rm Ext}}

\newcommand{\Cotor}{\operatorname{Cotor}}

\newcommand{\id}{{\rm id}}

\newcommand{\sll}{s}
\newcommand{\tl}{t}

\newcommand{\due}[3]{{}_{{#2 \!}} {#1}_{{\! #3}}\,}    % Zweifachindex

\newcommand{\qttr}[5]{{}^{{#2 \!}}_{{#4 \!}} {#1}^{#3}_{{\! #5}}}    % Vierfachindex 
\newcommand{\pl}{\partial}

\newcommand{\rmref}[1]{{\rm (}\ref{#1}{\rm )}}

\newcommand{{\Hl}}{{H^{\ell}}} 
\newcommand{{\mHop}}{{m_{H^{\rm op}}}} 
\newcommand{{\Hop}}{{H^{\rm op}}} 
\newcommand{{\mUop}}{{m_{U^{\rm op}}}} 
\newcommand{{\Uop}}{{U^{\rm op}}}
\newcommand{{\mVop}}{{m_{V^{\rm op}}}} 
\newcommand{{\Vop}}{{V^{\rm op}}}  
\newcommand{{\Ae}}{{A^{\rm e}}}
\newcommand{{\Be}}{{B^{\rm e}}}
\newcommand{{\Aop}}{{A^{\rm op}}}
\newcommand{{\Aope}}{({A^{\rm op}})^{\rm e}}
\newcommand{{\Aopl}}{{A^{\rm op}_\pl}}

\newcommand{{\Bop}}{{B^{\rm op}}}
\newcommand{{\Bope}}{({B^{\rm op}})^{\rm e}}
\newcommand{{\Bpl}}{{B_\pl}}

\newcommand{{\op}}{{{\rm op}}}
\newcommand{{\coop}}{{{\rm coop}}}
\newcommand{{\sop}}{{*^{\rm op}}}

\newcommand{\kmod}{k\mbox{-}\mathbf{Mod}}                     %
\newcommand{\amod}{A\mbox{-}\mathbf{Mod}}                     %
\newcommand{\amoda}{A^{\rm e}\mbox{-}\mathbf{Mod}}                  %
\newcommand{\umod}{U\mbox{-}\mathbf{Mod}}                     %  Modul-Kategorien
\newcommand{\modu}{U^\mathrm{op}\mbox{-}\mathbf{Mod}}         %
\newcommand{\ucomod}{U\mbox{-}\mathbf{Comod}}         %

\newcommand{\lact}{{\,\raise1pt\hbox{$\scriptscriptstyle{\rhd}$} \, }}                  %
\newcommand{\ract}{{\,\raise1pt\hbox{$\scriptscriptstyle{\lhd}$} \, }}                  %  A-Wirkungen
\newcommand{\blact}{{\,\raise1pt\hbox{$\scriptscriptstyle{\blacktriangleright}$} \, }}  %
\newcommand{\bract}{{\,\raise1pt\hbox{$\scriptscriptstyle{\blacktriangleleft}$} \, }}   %

\newcommand{{\gog}}{{G \rightrightarrows G_0}}

\newcommand{{\rra}}{\rightrightarrows}

\newcommand{{\lra}}{\ \longrightarrow \ }
\newcommand{{\lla}}{\ \longleftarrow \ }
\newcommand{{\lma}}{\ \longmapsto \ }

\def\kasten#1{\mathop{\mkern0.5\thinmuskip
\vbox{\hrule
      \hbox{\vrule
            \hskip#1
            \vrule height#1 width 0pt
            \vrule}%
      \hrule}%
\mkern0.5\thinmuskip}}

\newcommand{\bx}{{\kasten{6pt}}}

\newcommand{{\bull}}{{\scriptscriptstyle{\bullet}}}

\newcommand{{\qqquad}}{{\quad\quad\quad}}
\newcommand{\Aopp}{{\scriptscriptstyle{\Aop}}}

\newcommand{\Aee}{{\scriptscriptstyle{\Ae}}}

\begin{document}
%\date{results presented on 7 October 2010 in the S\'eminaire Topologie Alg\'ebrique, Paris
 % XIII}

\title{Cyclic structures in algebraic (co)homology theories} 

\author{Niels Kowalzig} 
\author{Ulrich Kr\"ahmer} 

\address{N.K.: Institut des Hautes \'Etudes Scientifiques, 
Le Bois-Marie,
35, route de Chartres, 
91440 Bures-sur-Yvette, 
France}

\email{kowalzig@ihes.fr}

\address{U.K.: University of Glasgow,
School of Mathematics \& Statistics, University 
Gardens, Glasgow G12 8QW, Scotland}

\email{Ulrich.Kraehmer@glasgow.ac.uk}

\begin{abstract}
This note discusses the
cyclic cohomology of a
left Hopf algebroid ($\times_A$-Hopf
algebra) with coefficients in 
a right module-left comodule, defined 
using a   
straightforward generalisation of the original
operators given by Connes and
Moscovici for Hopf algebras. 
Lie-Rinehart homology is a special case of this theory.
A generalisation of cyclic duality 
that makes sense for arbitrary para-cyclic objects yields a dual homology theory. 
The twisted cyclic homology of an associative algebra provides an example of this dual theory 
that uses coefficients that are not necessarily stable anti Yetter-Drinfel'd~modules.
\end{abstract}
\maketitle

%\tableofcontents

\section{Introduction}
\subsection{Topic} 
A left Hopf algebroid ($\times_A$-Hopf
algebra) $U$  is roughly speaking  
a Hopf algebra whose ground ring is not
a field $k$ but a possibly
noncommutative $k$-algebra $A$ 
\cite{Boe:HA,Schau:DADOQGHA}. The concept 
provides in particular a natural
framework for unifying and extending
classical constructions in homological
algebra. Group, Lie
algebra, Hochschild, and Poisson 
homology are all special cases of 
Hopf algebroid homology 
$$
		  H_\bull(U,M):=\mathrm{Tor}^U_\bull(M,A),\quad
%		  H^\bull(U,N):=\mathrm{Ext}^\bull_U(A,N),
		  M \in \modu,
$$
since the
rings $U$ over which these
theories can be expressed as
derived functors are all left Hopf
algebroids. This allows one for example to study cup and
cap products as well as the phenomenon
of Poincar\'e duality in a uniform way
\cite{KowKra:DAPIACT}.

Similarly, we describe here how 
the additional structure of a left
$U$-comodule on $M$ 
induces a para-cyclic structure 
(cf.~Section~\ref{defilambda}) on the canonical
chain complex $C_\bull(U,M)$
that computes $H_\bull(U,M)$
assuming $U$ is flat over
$A$. This
defines in particular
an analogue of
the Connes-Rinehart-Tsygan differential
$$
		  B : H_\bull(U,M)
		  \rightarrow H_{\bull-1}(U,M).
$$ 
Assuming a suitable compatibility
between the $U$-action and the
$U$-coaction (namely that $M$ is a
stable anti Yetter-Drinfel'd module), the
para-cyclic $k$-module $C_\bull(U,M)$ is
in fact cyclic and hence turned by $B$ into a
mixed complex. However, we will also
discuss concrete examples which
demonstrate the necessity to go beyond this
setting.  
 
\subsection{Background} 
The operator $B$ has been defined
by Rinehart on the
Hochschild homology of a commutative
$k$-algebra $A$ (with $M=A$ and $U=\Ae=A
\otimes_k A^\op$) in order to define the
De Rham cohomology of an arbitrary
affine scheme over $k$ \cite{Rin:DFOGCA}. 
Connes and
Tsygan independently rediscovered it
around 1980 as a central ingredient in
their definition of cyclic homology 
which 
extends Rinehart's
theory to noncommutative algebras 
\cite{Con:CCEFE,FeiTsy:AKT}.

Connes and
Moscovici, and Crainic initiated the study of
the case of a Hopf
algebra $U$ over $A=k$ with
one-dimensional coefficients $M$
\cite{ConMos:HACCATTIT,Cra:CCOHA}.
The class of admissible coefficient modules $M$
was subsequently enlarged to stable anti
Yetter-Drinfel'd modules  
\cite{HajKhaRanSom:SAYDM}, and Kaygun
finally obtained the construction for 
Hopf algebras with arbitrary modules-comodules as coefficients \cite{Kay:BCHWC, Kay:TUHCT}.

Noncommutative base rings appeared
for the
first time in the particular example of
the ``extended'' Hopf algebra governing
the transversal geometry of foliations
\cite{ConMos:DCCAHASITG}. The general
theory has then been further developed in 
\cite{BoeSte:ACATCD, BoeSte:CCOBAAVC,
HasRan:EHGEAHCC,KhaRan:PHAATCC, Kow:HAATCT,
KowPos:TCTOHA, Mas:NCGTMC}. 

\subsection{Results}
Our first aim here is to give explicit
formulas for the most
straightforward generalisation
of the original operators defined by Connes and
Moscovici in \cite{ConMos:DCCAHASITG}
towards Hopf algebroids
and completely general coefficients. We
copy the result here, see the main text
for the details and in particular for
the notation used:
\begin{thm}\label{main}
Let $U$ be a left Hopf algebroid over
a $k$-algebra $A$, and $M$ be a right $U$-module and left
$U$-comodule with compatible induced left $A$-module structures. Then
$
		  C^\bull(U,M):=U^{\otimes_A
		  \bull} \otimes_A M
$ 
carries a canonical
para-cocyclic $k$-module structure with
codegeneracies and cofaces
$$
\begin{array}{rll}
\gd_i(z \otimes_\ahha m) \!\!\!\!&= \left\{\!\!\!
\begin{array}{l} 1 
\otimes_\ahha u^1 \otimes_\ahha \cdots
 \otimes_\ahha u^n \otimes_\ahha m  
\\ 
u^1 \otimes_\ahha \cdots \otimes_\ahha \Deltaell (u^i) \otimes_\ahha \cdots
 \otimes_\ahha u^n \otimes_\ahha m
\\
u^1 \otimes_\ahha \cdots \otimes_\ahha u^n \otimes_\ahha m_{(-1)} \otimes_\ahha m_{(0)} 
\end{array}\right. \!\!\!\!\!\!\!\! 
& \!\! \begin{array}{l} \mbox{if} \ i=0, \\ \mbox{if} \
  1 \leq i \leq n, \\ \mbox{if} \ i = n + 1,  \end{array} \\
\gd_j(m) \!\!\!\! &= \left\{ \!\!\!
\begin{array}{l}
		  1 
		  \otimes_\ahha m  \quad
\\
m_{(-1)} \otimes_\ahha m_{(0)}  \quad 
\end{array}\right. & \!\!
\begin{array}{l} \mbox{if} \ j=0, \\ \mbox{if} \
  j = 1 ,  \end{array} \\
\gs_i(z \otimes_\ahha m) \!\!\!\! 
&= u^1 \otimes_\ahha \cdots \otimes_\ahha
\eps (u^{i+1}) \otimes_\ahha \cdots
\otimes_\ahha u^n \otimes_\ahha m &
\!\!
\begin{array}{l} 0 \leq i \leq n-1,\end{array} 
\end{array}
$$
and cocyclic operator
$$
\tau_n(z \otimes_\ahha m) = u^1_{-(1)}u^2 \otimes_\ahha \cdots
 \otimes_\ahha
 u^1_{-(n-1)}u^n\otimes_\ahha
 u^1_{-(n)}m_{(-1)} \otimes_\ahha
 m_{(0)}u^1_+, 
$$
where we abbreviate $z:=u^1 \otimes_\ahha
\cdots \otimes_\ahha u^n$. 
\end{thm}

The proof follows closely the 
literature cited above,
which contains similar constructions
of a large variety of 
para-cyclic and para-cocyclic modules
assigned to Hopf algebroids.
Many of these are
related by various dualities 
($k$-linear duals, $\Tor$
vs.~$\Ext$ vs.~$\Cotor$,
dual Hopf algebroids when applicable,
and cyclic
duality). However, there seems no
reference for the exact setting we
consider here. Also, Kaygun's 
pivotal observation mentioned above
seems a little lost in the references
working over noncommutative base
algebras. Last but not
least, the above 
answers also the question of how
the Hopf-cyclic (co)homologies in
\cite{Kow:HAATCT, KowPos:TCTOHA}  
can be extended to general coefficients.

Secondly, it has been pointed out 
by several authors that the standard operation of
cyclic duality which canonically identifies 
cyclic and cocyclic objects 
does not lift 
to para-(co)cyclic objects, see e.g.~\cite{BoeSte:ACATCD}. 
However, we show in
Section~\ref{wirdschon} that a different
choice of
anti-autoequivalence of the cyclic
category leads to a form 
of cyclic duality that does lift. This
allows us to construct in full
generality a cyclic dual  
$(C_\bull(U,M),d_\bull,s_\bull,t_\bull)$ 
from the para-cocyclic module from
Theorem~\ref{main}. 
We provide an isomorphism of this
with the para-cyclic module 
$M \otimes_\Aopp  
		  (\due U \blact \ract)^{\otimes_\Aopp  \bull}$ 
whose structure maps are given by
$$
\!\!\!
\begin{array}{rcll}
d_i(m \otimes_\Aopp x)  &\!\!\!\!\! =& \!\!\!\!\!
\left\{ \!\!\!
\begin{array}{l}
m \otimes_\Aopp u^1  \otimes_\Aopp   \cdots   \otimes_\Aopp   \eps(u^n) \blact u^{n-1}
\\
m \otimes_\Aopp \cdots \otimes_\Aopp  u^{n-i} u^{n-i+1}
 \otimes_\Aopp  \cdots 
\\
mu^1 \otimes_\Aopp u^2  \otimes_\Aopp   \cdots    \otimes_\Aopp  
u^n 
\end{array}\right.  & \!\!\!\!\!\!\!\!\!\!\!\! \,  \begin{array}{l} \mbox{if} \ i \!=\! 0, \\ \mbox{if} \ 1
\!  \leq \! i \!\leq\! n-1, \\ \mbox{if} \ i \! = \! n, \end{array} \\
s_i(m \otimes_\Aopp x) &\!\!\!\!\! =&\!\!\!\!\!  \left\{ \!\!\!
\begin{array}{l} m  \otimes_\Aopp   u^1  
\otimes_\Aopp   \cdots   \otimes_\Aopp
 u^n  \otimes_\Aopp 
		  1
\\
m \otimes_\Aopp \cdots \otimes_\Aopp   u^{n-i} 
\otimes_\Aopp   1  
\otimes_\Aopp   u^{n-i+1}  \otimes_\Aopp
 \cdots  
\\
m \otimes_\Aopp 1 
\otimes_\Aopp u^1  \otimes_\Aopp   \cdots    \otimes_\Aopp  u^n 
\end{array}\right.   & \!\!\!\!\!\!\!\!\!\!\!  \begin{array}{l} 
\mbox{if} \ i\!=\!0, \\ 
\mbox{if} \ 1 \!\leq\! i \!\leq\! n-1, \\  \mbox{if} \ i\! = \!n, \end{array} \\
t_n(m \otimes_\Aopp x) 
&\!\!\!\!\!=&\!\!\!\!\! 
m_{(0)} u^1_+ \otimes_\Aopp u^2_+ \otimes_\Aopp  \cdots  \otimes_\Aopp u^n_+ \otimes_\Aopp u^n_- \cdots u^1_- m_{(-1)},  
& \\
\end{array}
$$
where we abbreviate $x:=u^1
\otimes_\Aopp \cdots \otimes_\Aopp u^n$.

It is precisely this variation of
Hopf-cyclic theory that has the ordinary
Hopf algebroid homology as underlying
simplicial homology, and in particular
the one which
reduces to the original cyclic homology
of an associative algebra when
one applies it to the Hopf algebroid $U=\Ae$. 
Now the freedom to consider arbitrary
coefficients becomes crucial, since it
allows one for example to incorporate
the twisted cyclic homology of
Kustermans, Murphy and Tuset 
\cite{KusMurTus:DCOQGATCC}. 
That paper has been the first one to generalise the Connes-Rinehart-Tsygan operator $B$ on the 
Hochschild homology of an associative algebra to coefficients in $(A,A)$-bimodules other than $A$ itself, namely those where one of the two 
actions of $A$ 
on itself is twisted by an algebra automorphism $\sigma$. 
When viewed as a special case of the above Hopf-cyclic homology, these coefficients are not stable anti Yetter-Drinfel'd, 
and one sees that an $\Ae$-comodule structure is all one needs to define $B$.
We discuss this example in the last
section of the
paper, and also the example of Lie-Rinehart
homology which is an important classical case of the cyclic cohomology theory from Theorem~\ref{main}.

N.K.~is supported by an I.H.\'E.S. visiting grant. 
U.K.~is supported by the EPSRC 
fellowship EP/E/043267/1 and partially by the Marie Curie PIRSES-GA-2008-230836
network.   

\section{Preliminaries}

\subsection{Some conventions}
Throughout this note, ``ring'' means 
``unital and associative ring'', and we fix   
a commutative ring $k$.
All other algebras, modules etc.~will 
have an underlying structure of a
$k$-module. Secondly, we fix a
$k$-algebra $A$, i.e.\ a ring with a 
ring homomorphism 
$ \eta_\ahha : k \rightarrow Z(A)$ to  
its centre.
We denote by $\amod$ the category of 
left $A$-modules, 
by $A^\mathrm{op}$ the 
opposite and by
$A^\mathrm{e} := A \otimes_k A^\mathrm{op}$ 
the enveloping algebra 
of $A$. 
An {\em $A$-ring} is a monoid
in the monoidal
category $(\amoda, \otimes_A, A)$ of $\Ae$-modules (i.e.\ $(A,A)$-bimodules
with symmetric action of $k$), fulfilling
associativity and unitality. Likewise,  
an {\em $A$-coring} is a comonoid in
$(\amoda, \otimes_A, A)$,  
fulfilling coassociativity and
counitality. 

Our main object is an $\Ae$-ring $U$
(a monoid in $(\Ae \otimes_k
\Ae)\mbox{-}\mathbf{Mod}$). Explicitly,
such an $\Ae$-ring is given by a $k$-algebra 
homomorphism 
$
		  \eta=\eta_U : A^\mathrm{e} \rightarrow U
$ 
whose restrictions
$$
		  \sll:= \eta( - \otimes_k 1) : 
		  A \to U 
		  \quad \mbox{and} \quad 
		  \tl := \eta(1 \otimes_k -) : 
		  \Aop \to U
$$
will be called the {\em source} and {\em
target} map. Left and right multiplication in $U$
give rise to an 
$(\Ae,\Ae)$-bimodule structure on $U$,
that is, four commuting actions of $A$
that we denote by 
\begin{equation}\label{bimod-lmod}
		  a \lact u \ract b :=\sll(a)t(b)u,\quad 
		  a \blact u \bract b
		  :=u\sll(b)t(a), 
		  \quad a,b\in A, \ u \in U.
\end{equation}
If not stated otherwise, we view $U$ as
an $(A,A)$-bimodule using the actions $\lact,\ract$.
In particular, we define 
the tensor product
$U \otimes_A U$
with respect to this bimodule structure.
On the other hand, using the
actions $\blact, \bract$ permits to define the 
{\em Takeuchi product}
\begin{equation}
\label{taki}
		  U \times_A U :=  
		  \{\textstyle\sum_i u_i \otimes_A
		  v_i 
		  \in U \otimes_A U \mid 
		  \sum_i a \blact u_i 
		  \otimes_A v_i = 
		  \sum_i u \otimes_A v_i \bract a, 
		  \ \forall a \in A\}.
\end{equation}
This is an
$\Ae$-ring via factorwise
multiplication. Similarly,
$\mathrm{End}_k(A)$ is an $\Ae$-ring
with ring structure given by composition
and $(A,A)$-bimodule structure 
$(a \varphi b)(c):=\varphi (bca)$,
$\varphi \in \mathrm{End}_k(A)$, $a,b,c
\in A$.

\subsection{Bialgebroids} {\cite{Tak:GOAOAA}}
Bialgebroids are a generalisation of
bialgebras. An important subtlety 
is that the algebra and coalgebra 
structure are defined in
different monoidal categories. 

\begin{definition}\label{left-bialg}
Let $A$ be a $k$-algebra.
A {\em left bialgebroid} over $A$ 
(or {\em $A$-bialgebroid} or 
{\em $\times_A$-bialgebra}) is an $\Ae$-ring $U$
together with two homomorphisms of
 $\Ae$-rings
$$
		  \Deltaell : U \rightarrow U
		  \times_A U,\quad
		  \hat \eps : U \rightarrow \mathrm{End}_k(A)  
$$
which turn $U$ into an $A$-coring 
with coproduct $\Deltaell$ (viewed as a
 map $U \rightarrow U \otimes_A U$) and
 counit $\eps : U \rightarrow A$,
$u \mapsto (\hat\eps (u))(1)$.
\end{definition}

Note that this means for example that
$\eps$ satisfies for all $u, v\in U$
\begin{equation*}\label{leftcou}
\eps(uv) = \eps(u \bract \eps (v)) =   \eps(\eps (v) \blact u).
\end{equation*}

Analogously one defines {\em right} 
bialgebroids where the roles of
$\lact,\ract$ and $\blact,\bract$ are
exchanged. 
We shall not write out the details, but rather refer
to \cite{KadSzl:BAODTEAD,Boe:HA}. 

\subsection{Left Hopf algebroids} {\cite{Schau:DADOQGHA}}
Left Hopf algebroids have been
introduced by Schauenburg 
under the name {\em $\times_A$-Hopf
algebras} and 
generalise Hopf 
algebras towards left bialgebroids.
For a left bialgebroid $U$ over $A$, one
defines the
{\em (Hopf-)Galois map} 
\begin{equation}
\label{Galois}
\beta: {}_\blact U \otimes_\Aopp U_\ract \to 
U_\ract \otimes_A {}_\lact U, \quad u
\otimes_\Aopp v \mapsto  
u_{(1)}  \otimes_A u_{(2)}v, 
\end{equation}
where 
\begin{equation}
\label{tata}
 {}_\blact U \otimes_\Aopp U_\ract = U
 \otimes_k U/
{{\rm
 span}\{a \blact u \otimes_k v - u
 \otimes_k v \ract a\,|\,
u,v \in U, a \in A \}}.
\end{equation}

\begin{dfn}\cite{Schau:DADOQGHA}
\label{hopftimesleft}
A left $A$-bialgebroid $U$ is 
called a 
{\em left Hopf algebroid} 
(or {\em $\times_A$-Hopf algebra}) if
$\beta$ is a bijection. 
\end{dfn}

In a similar manner, one
 defines {\em right Hopf algebroids} (cf.\ \cite[Prop.\ 4.2]{BoeSzl:HAWBAAIAD}). 

\noindent Following \cite{Schau:DADOQGHA}, we
adopt a Sweedler-type notation 
\begin{equation}\label{pm}
		  u_+ \otimes_\Aopp u_- := 
		  \beta^{-1}( u \otimes_A 1)
\end{equation}
for the so-called {\em translation map}
$
		  \beta^{-1}(- \otimes_A 1) : U \rightarrow 
		  {}_\blact U \otimes_\Aopp U_\ract.
$ 
Useful for our subsequent calculations, 
one has for all $u, v \in U$, $a \in A$
\cite[Prop.~3.7]{Schau:DADOQGHA}:
\begin{eqnarray}
\label{Sch1}
u_{+(1)} \otimes_A u_{+(2)} u_- &=& u \otimes_A 1 \in U_\ract \otimes_A {}_\lact U, \\
\label{Sch2}
u_{(1)+} \otimes_\Aopp u_{(1)-} u_{(2)}  &=& u \otimes_\Aopp  1 \in  {}_\blact U
\otimes_\Aopp  U_\ract, \\ 
\label{Sch3}
u_+ \otimes_\Aopp  u_- & \in 
& U \times_\Aop U, \\
\label{Sch38}
u_{+(1)} \otimes_A u_{+(2)} \otimes_\Aopp  u_{-} &=& u_{(1)} \otimes_A
u_{(2)+} \otimes_\Aopp  u_{(2)-},\\
\label{Sch37}
u_+ \otimes_\Aopp  u_{-(1)} \otimes_A u_{-(2)} &=& u_{++} \otimes_\Aop
u_- \otimes_A u_{+-}, \\
\label{Sch4}
(uv)_+ \otimes_\Aopp  (uv)_- &=& u_+v_+
\otimes_\Aopp  v_-u_-, 
\\ 
\label{Sch47}
u_+u_- &=& \sll (\eps (u)), \\
\label{Sch48}
u_+ \tl (\eps (u_-)) &=& u, \\
\label{Sch5}
(\sll (a) \tl (b))_+ \otimes_\Aopp  (\sll (a) \tl (b) )_- 
&=& \sll (a) \otimes_\Aopp  \sll (b), 
\end{eqnarray}
where in (\ref{Sch3}) we mean the Takeuchi product
\begin{equation*}
\label{petrarca}
		  U \times_\Aop U:=
		  \left\{\textstyle\sum_i u_i \otimes_\Aopp  v_i \in 
		  {}_\blact U \otimes_\Aopp  U_\ract\,|\,
		  \sum_i u_i \ract a \otimes_\Aopp  v_i=
		  \sum_i u_i \otimes_\Aopp  a \blact
		  v_i
		  \right\},
\end{equation*}
which is an algebra by factorwise
multiplication, but with opposite 
multiplication on the second factor.
Note that in (\ref{Sch37}) the tensor product
over $A^\mathrm{op}$ links the first and
third tensor component. 
By (\ref{Sch1}) and (\ref{Sch3})
one can write 
\begin{equation}
\label{pmb}
\beta^{-1}(u \otimes_A v) = u_+ \otimes_\Aopp  u_-v,
\end{equation}
 which is easily checked to be
 well-defined over $A$ with (\ref{Sch4}) 
 and (\ref{Sch5}).

\begin{rem}
Observe that there is no notion of
antipode for a left Hopf algebroid. B{\"o}hm and Szlach{\'a}nyi  
have introduced the concept of a ({\em full} or {\em two-sided}) {\em
Hopf algebroid} \cite{Boe:HA}, which is, roughly speaking, an
algebra equipped with a left and a right 
bialgebroid structure over
anti-isomorphic base algebras $A$ and
$B$, together with an antipode mapping
from the left bialgebroid to the
right. However, it is proved in 
\cite[Prop.\ 4.2]{BoeSzl:HAWBAAIAD} 
that a
full Hopf algebroid with invertible 
antipode can be equivalently described 
as an algebra with both a left and a
right Hopf algebroid structure subject
to compatibility conditions, 
which motivates to speak of left Hopf algebroids
rather than $\times_A$-Hopf algebras.  
\end{rem}

\subsection{$U$-modules}
Let $U$ be a left bialgebroid with structure maps as before. Left and right $U$-modules are defined as modules over the ring $U$, with respective actions denoted by juxtaposition or, at times, by a dot for the sake of clarity. 
We denote the respective categories by
$\umod$ and $\modu$; while $\umod$ is
a monoidal category, $\modu$ is in general
not \cite{Schau:BONCRAASTFHB}. One has a forgetful functor  
$\umod \rightarrow \amoda$ using which
we consider every left 
$U$-module $N$ also as an $(A,A)$-bimodule
with actions 
\begin{equation}
\label{brot}
	anb := 	  a \lact n \ract b := \sll(a)\tl(b)n,\quad
		  a,b \in A,n \in N.
\end{equation} 
Similarly, every right $U$-module $M$ is also 
an $(A,A)$-bimodule via
\begin{equation}\label{salz}
	amb:= 	  a \blact m \bract b := 
			m \sll(b) \tl(a),\quad
		  a,b \in A,m \in M,
\end{equation} 
and in both cases we usually prefer to express these actions just by juxtaposition if no ambiguity is to be expected.

\subsection{$U$-comodules}

Similarly as for coalgebras, one may
define comodules over 
bialgebroids, but the underlying
$A$-module structures need some extra
attention. For the following definition
confer e.g.\
\cite{Schau:BONCRAASTFHB,Boe:GTFHA,
BrzWis:CAC}.

\begin{dfn}
\label{tempo}
A {\em left $U$-comodule} for a left bialgebroid $U$ over $A$
is a left comodule of the underlying $A$-coring $(U, \Deltaell,
\eps)$, 
i.e.\ a left $A$-module $M$ with action $L_\ahha: (a,m) \mapsto am$ and a
left $A$-module map  
\begin{equation*}
%\label{ucomod}
		  \gD_\emme: 
		  M \to U_\ract \otimes_A M, 
		  \quad 
		  m \mapsto m_{(-1)} \otimes_A m_{(0)}
\end{equation*}
%where 
%$$
%U \otimes_A M  := U \otimes_k M/{{\rm span}\{\tl a u  \otimes m - u
%\otimes am \mid a \in A\}},
%$$
satisfying the usual coassociativity and
 counitality axioms 
$$
		  (\Deltaell \otimes \id) \circ \gD_\emme
		  =  
		  (\id \otimes \gD_\emme) \circ \gD_\emme	  
		  \quad \mbox{and} \quad 
		  L_\ahha \circ (\eps \otimes \id)
		  \circ \gD_\emme =  \id.
$$
%
% For two left $U$-comodules $M, M'$, we
% denote the set of {\em left
% $U$-comodule morphisms} by
% $$
%		  \Com_{U}(M,M') := 
%		  \{ \psi \in \Hom_A(M,M') \mid (\id \otimes \psi)
%		  \circ \gD_\emme =
%		  \gD_{\scriptscriptstyle{M'}}
%		  \circ \psi \}, 
% $$
% and 
We denote the category of left
 $U$-comodules 
by $\ucomod$.
\end{dfn}
Analogously one defines {\em
right} $U$-comodules and comodules for
right bialgebroids.  

\noindent On any left $U$-comodule
one can additionally define 
a right $A$-action 
\begin{equation}\label{grimm}
ma := \eps\big(m_{(-1)}\sll (a)\big)m_{(0)}.
\end{equation}
This is the unique action that turns $M$ into an 
$\Ae$-module in such a way that the coaction is an
$\Ae$-module morphism
$$
		  \gD_\emme: M \to U \times_A M, 
$$
where $U \times_A M$ is the Takeuchi product
$$
U \times_A M := \{\textstyle\sum_i u_i \otimes_A m_i		 
		  \in U \otimes_A M \mid
		  \sum_i u_i\tl (a) \otimes_A m_i
		  = \sum_i u_i \otimes_A m_i a, \
		  \forall a \in A\}. 
$$
As a result, $\Delta_\emme$ satisfies the identities 
\begin{eqnarray}
\label{maotsetung}
		  \gD_\emme(amb) &=& 
		  \sll (a)  m_{(-1)} \sll (b) \otimes_A
		  m_{(0)}, \\
		  \label{douceuretresistance}
		  m_{(-1)} \otimes_A m_{(0)}a
		  &=&  
		  m_{(-1)}\tl (a) \otimes_A m_{(0)}.
\end{eqnarray}
This is compatible
with (\ref{grimm}) since 
one has $\eps(u \sll(a))=\eps(u
\tl(a))$ for all $u \in U,a \in A$.

One can then prove (see \cite[Thm.\ 3.18]{Boe:HA} and
\cite[Prop.\ 5.6]{Schau:BONCRAASTFHB}) that
$\ucomod$ has a monoidal structure
such that the forgetful 
functor $\ucomod \to \amoda$ is
		  monoidal:
for any two comodules $M, M'
\in \ucomod$, their tensor product $M \otimes_A M'$ is a left
$U$-comodule by means of the coaction
 \begin{equation*}
%\label{ComodL}
\begin{array}{rcl}
\gD_{\scriptscriptstyle{M \otimes_A M'}} {}: M \otimes_A M' &\to& U \otimes_A (M
\otimes_A M'), \\
 m \otimes_A m' &\mapsto& m_{(-1)} m'_{(-1)} \otimes_A m_{(0)} \otimes_A
m'_{(0)}. 
\end{array}
\end{equation*}
The map $\gD_{\scriptscriptstyle{M \otimes_A M'}} {}$ is
easily checked to be well-defined.

\begin{rem}
If $\gs \in U$ is a grouplike element in
a (left) bialgebroid, then 
\begin{equation*}
%\label{Comoda}
		  \qttr \gD {}{}\ahha{} (a) :=
		  \tl(a) \gs 
		  \quad \hbox{and} 
		  \quad 
		  \gD_\ahha (a) := 
		  \sll(a)\gs, 
		  \qquad a \in A, 
\end{equation*}
define right and left $U$-comodule
 structures on $A$, which we shall 
refer to as induced by $\gs$.
In particular, the 
base algebra $A$ carries for any
bialgebroid both a canonical right
and a canonical left coaction induced by
$\sigma=1$, 
contrasting the fact that $A$ carries
in general only a canonical 
left $U$-module structure induced
by $\eps$, but no right one.
\end{rem}

\begin{rem}
\label{cannibale}
A special feature for bialgebroids $U$ over commutative base algebras $A$ with $s=t$ 
is that every left $A$-module $M$ can be made into a, say, left 
$U$-comodule by means of the trivial coaction $m \mapsto 1 \otimes_A m$ 
(it follows from \rmref{maotsetung} that this is not possible in general).
\end{rem}

\subsection{Stable anti Yetter-Drinfel'd modules}

The following definition is the left
bialgebroid right module and left
comodule version of the corresponding notion in
\cite{BoeSte:CCOBAAVC}. For
Hopf algebras, the concept  
goes back to \cite{HajKhaRanSom:SAYDM}. 
\begin{dfn}
\label{SAYD}
Let $U$ be a left Hopf algebroid 
with structure maps as before, and let
 $M$ simultaneously be a left
 $U$-comodule with coaction denoted as
 above and a right $U$-module with
 action denoted by $(m, u) \mapsto mu$
 for $u \in U$, $m \in M$. We call $M$
 an {\em anti Yetter-Drinfel'd (aYD)
 module} provided the following holds: 
\begin{enumerate}
\item
The $\Ae$-module structure on $M$ originating from its nature as $U$-comodule coincides with the 
$\Ae$-module structure induced by the right $U$-action on $M$, i.e., for all $a, b \in A$ and $m \in M$ we have
\begin{equation}
\label{campanilla}
amb = a \blact m \bract b,
\end{equation}
where the right $A$-module structure on the left hand side is given by \rmref{grimm}. 
\item
For $u \in U$ and $m \in M$ one has
\begin{equation}
\label{huhomezone}
		  \gD_\emme(mu) = 
		  u_- m_{(-1)} u_{+(1)} \otimes_A m_{(0)} u_{+(2)}.
\end{equation}
\end{enumerate}
The anti Yetter-Drinfel'd module $M$ is
 said to be {\em stable (SaYD)} if for
 all $m \in M$ one has
$$
m_{(0)}m_{(-1)} = m.
$$
\end{dfn}

\begin{rem}
Observe that it is not obvious
that the expression on the right
hand side of (\ref{huhomezone}) makes
sense, but this follows 
from \rmref{taki}, \rmref{Sch3},
and \rmref{douceuretresistance}.  
\end{rem}

\subsection{Cyclic (co)homology}\label{defilambda}
We will not recall the formalism of
cyclic (co)homology in full detail. However, since this
notion is not contained in our standard
reference \cite{Lod:CH} we recall that 
para-(co)cyclic $k$-modules generalise 
(co)cyclic $k$-modules by dropping the
condition that the cyclic operator
implements an action of
$\mathbb{Z}/(n+1)\mathbb{Z}$ on the
degree $n$ part. Thus a para-cyclic
$k$-module is a simplicial
$k$-module
$(C_\bull,d_\bull,s_\bull)$ 
and a para-cocyclic $k$-module is a
cosimplicial $k$-module
$(C^\bull,\delta_\bull,\sigma_\bull)$
together with $k$-linear maps 
$t_n : C_n \rightarrow C_n$ and 
$\tau_n : C^n \rightarrow C^n$
satisfying, respectively
\begin{equation}
\label{belleville}
\!\!\!\!\!\!\!\!
\begin{array}{cc}
\begin{array}{rcl}
d_i \circ t_n  &\!\!\!\!\!\!=&\!\!\!\!\!\! \left\{\!\!\!
\begin{array}{ll}
t_{n-1} \circ d_{i-1} 
& \!\!\!\! \mbox{if} \ 1 \leq i \leq n, \\
 d_n & \!\!\!\! \mbox{if} \
i = 0,
\end{array}\right. \\
s_i \circ t_n &\!\!\!\!\!\!=&\!\!\!\!\!\! \left\{\!\!\!
\begin{array}{ll}
t_{n+1} \circ s_{i-1} & \!\!\!\!
\mbox{if} \ 1 \leq i \leq n, \\
 t^2_{n+1} \circ s_n
 & \!\!\!\! \mbox{if} \
i = 0, \\
\end{array}\right.
\end{array}
\!\!\!\!&\!\!\!\!
\begin{array}{rcll}
\tau_n \circ \gd_i &\!\!\!\!\!=&\!\!\!\!\! \left\{\!\!\!
\begin{array}{l}
\gd_{i-1}\circ \tau_{n-1} \\
 \gd_n 
\end{array}\right. & \!\!\!\!\!\!\!\!\! \begin{array}{l} \mbox{if} \ 1 \leq i \leq n,
 \\ \mbox{if} \ i = 0,
 \end{array} \\
\tau_n \circ \sigma_i &\!\!\!\!\!=&\!\!\!\!\! \left\{\!\!\!
\begin{array}{l}
\sigma_{i-1} \circ \tau_{n+1} \\
 \sigma_n \circ \tau^2_{n+1} 
\end{array}\right. & \!\!\!\!\!\!\!\!\!
\begin{array}{l} \mbox{if} \ 1 \leq i
 \leq n,  
\\ \mbox{if} \ i = 0. \end{array} 
\end{array}
\end{array}
\end{equation}
It follows from these relations that
$t_n^{n+1}$ respectively $\tau_n^{n+1}$
commutes with all the (co)faces and
(co)degeneracies. Hence any
para-(co)cyclic $k$-module defines a
(co)cyclic one formed by the cokernels
of $\mathrm{id}_{C_n}-t_n^{n+1}$
respectively the kernels of $\mathrm{id}_{C^n}-\tau_n^{n+1}$.
The cyclic (co)homology of a
para-(co)cyclic $k$-module is defined
as the cyclic (co)homology of this
associated (co)cyclic $k$-module.

Just like (co)cyclic $k$-modules,
para-(co)cyclic ones can be viewed more
conceptually as functors 
$\gL^\op \to \kmod$ respectively
$\gL \to  \kmod$, where $\Lambda$ is the
appropriate covering of Connes' cyclic
category $\Lambda_1$. Hence as Connes' category, 
$\Lambda$ has objects $\{[n]\}_{n
\in \mathbb{N}}$ and the set of
morphisms has generators
obeying the same
relations except
for 
$\tau_n^{n+1}=\mathrm{id}_{[n]}$. The
localisation of this category at the set
of all $\tau_n$ has been studied already
by Fe{\u\i}gin and Tsygan in
\cite{FeiTsy:AKT} where it is denoted by
$\Lambda_\infty$. However, we stress
that in the present article $\tau_n$ is not
assumed to be an isomorphism. We will
call $\Lambda$ the para-cyclic category. 

\section{Hopf-Cyclic Cohomology with Coefficients}
\label{effacersvp}

\subsection{Para-cocylic
structures on corings}
Following \cite{Cra:CCOHA, BoeSte:CCOBAAVC} we
first define in this section 
an auxiliary para-cocyclic $k$-module 
that is relatively easy to
construct. 
For this, $U$ just needs to be
a left bialgebroid
and $M$ needs to be a left $U$-comodule.  
\begin{comment}
The tensor product
$
\due U {} \bract \otimes_A M 
$
%:= M \otimes_k U/{{\rm
% span}\{ma \otimes u - m \otimes \sll(a)u \mid a \in A\}}
with $(A,A)$-bimodule structure $a(u \otimes_A m)b := \sll a u  \otimes_A mb$ 
can be given a $U$-bicomodule structure by means of 
a right $U$-coaction
$$
\due U {} \bract \otimes_A M \to 
		  (\due U {} \bract \otimes_A M)
 \otimes_A U, \quad  
		  u_{(1)} \otimes_A u_{(2)},
$$
and a right one
$$
		  M \otimes_A U \to 
		  U \otimes_A (M \otimes_A U),
		  \quad 
		  m \otimes_A u \mapsto 
		  m_{(-1)}u_{(1)} \otimes_A m_{(0)} \otimes_A u_{(2)}.
$$
\end{comment}
Define then 
$$
		  \boehmcomplex^\bull(U,M) := 
		  U^{\otimes_A \bull+1} \otimes_\Aee
		  M,
$$
where $U$ is considered with the usual
$(A,A)$-bimodule structure given by $\lact,\ract$.
So $\boehmcomplex^\bull(U,M)$ is 
${\due  U \lact \ract}^{\otimes_A \bull+1} \otimes_k
		  M$
modulo the span of elements 
$$
		  \{ u^0 \otimes_\ahha \cdots \otimes_\ahha
		  u^n 
		  \otimes_\Aee amb - 
		  b \lact u^0 \otimes_\ahha \cdots
		  \otimes_\ahha u^n 
		  \ract a \otimes_\Aee m \mid a, b \in A \}. 
$$
Now define the following operators,
where we abbreviate $w:=u^0 \otimes_\ahha
\cdots \otimes_\ahha u^n$:  
\begin{equation}
\hspace*{-.4cm}\begin{array}{rcll}\label{duracell}
		  \gd'_i(w \otimes_\Aee m)  &\!\!\!\!\!=&\!\!\!\!\! \left\{\!\!\!
\begin{array}{l}
% c^0_{(1)} \otimes_\ahha c^0_{(2)} \otimes_\ahha c^1 \otimes_\ahha \cdots \otimes_\ahha
% c^n  \otimes_\ahha \\ 
		  u^0 \otimes_\ahha \cdots
			\otimes_\ahha 
		  \Deltaell (u^i) \otimes_\ahha
		  \cdots \otimes_\ahha u^n
		  \otimes_\Aee m \\ 
u^0_{(2)} \otimes_\ahha u^1 \otimes_\ahha \cdots \otimes_\ahha  m_{(-1)} u^0_{(1)} \otimes_\Aee m_{(0)} 
\end{array}\right. & \!\!\!\!\! \begin{array}{l} %\mbox{if} \ i = 0 \\ 
\mbox{if} \ 0 \leq i
  \leq n, \\ \mbox{if} \ i = n + 1, \end{array} \\
\sigma'_i(w \otimes_\Aee m )
 &\!\!\!\!\!=&\!\!\!\!\! u^0 \otimes_\ahha
 \cdots \otimes_\ahha \tl (\eps
 (u^{i+1})) u^i 
\otimes_\ahha \cdots \otimes_\ahha u^n
\otimes_\Aee m & 
\!\! 0 \leq i \leq n-1, \\
\tau'_n(w \otimes_\Aee m) 
&\!\!\!\!\!=&\!\!\!\!\!  u^1 \otimes_\ahha
\cdots \otimes_\ahha u^n \otimes_\ahha m_{(-1)} u^0 \otimes_\Aee m_{(0)}, &
\end{array}
\end{equation}
which are shown to be well-defined using
the Takeuchi condition
for $\gD_\emme$. The following is
checked in a straightforward manner: 

\begin{lem}
\label{nepaseffacer}
 The operators
 $(\gd'_\bull, \gs'_\bull, \tau'_\bull)$
turn $\boehmcomplex^\bull(U,M)$
into a para-cocyclic $k$-module. 
\end{lem}

\subsection{The quotient $\boehmcomplex^\bull(U,M) 
\rightarrow C^\bull(U,M)$}
The para-cocyclic $k$-module that defines 
Hopf-cyclic cohomology is
the canonical quotient
$$
		  U^{\otimes_A \bull +1} \otimes_\Uop M
$$
of $\boehmcomplex^\bull(U,M)=
U^{\otimes_A \bull+1} \otimes_\Ae M$ defined
above. This quotient makes sense whenever $M$
also carries a right
$U$-module structure that induces the
same $\Ae$-module structure as the left
$U$-coaction, see \rmref{campanilla}.
In the next section we will discuss that the para-cocyclic structure 
of $\boehmcomplex^\bull(U,M)$ 
descends to
this quotient.
%assuming the anti Yetter-Drinfel'd condition.
However, 
for the applications in noncommutative
geometry one rewrites 
the resulting para-cocyclic $k$-module
so that the object (but not the cocyclic operator) 
takes an easier form,
and in the present
section we construct the involved isomorphism.

%\begin{dfn}
%We call the functor 
%$$
%		  \cdot \otimes_U A : \modu \rightarrow \kmod
%$$
%the functor taking the {\em coinvariants}
%$M \otimes_U A$ of $M \in \modu$.
%\end{dfn}
%In our case, we will 
%consider the coinvariants of right
%modules of a specific form: 
Recall (e.g.~from \cite[Lem.\
3]{KowKra:DAPIACT}) 
that if $U$ is a left Hopf algebroid, 
then the tensor
product $N \otimes_A M$ of 
$M \in \modu$, $N \in \umod$
 (considered 
with the $(A,A)$-bimodule
structures \rmref{brot} and
\rmref{salz}) 
carries a right $U$-module structure with action
$$
		  (n \otimes_A m)u := 
		  u_-n \otimes_A mu_+,
$$
and hence using \rmref{brot} and \rmref{salz}
becomes an $(A, A)$-bimodule by
$$
		  a \blact (n \otimes_A m) \bract
		  b 
		  := 
		  \big(n \otimes m\big)\tl (a)
		  \sll (b)
		  = \sll (a) n \otimes_A m 
		  \sll (b) = a \lact n \otimes_A m \bract b,
$$
where in the second equation \rmref{Sch5} was used.

Now observe that on a right $U$-module
of this form, the coinvariant functor  
$$
		  - \otimes_U A : \modu
		  \rightarrow \kmod
$$
takes a particularly simple form:
\begin{lemma}\label{eli}
If $U$ is a left Hopf algebroid, then
for all  
$M \in \modu$, $N \in \umod$ there is a
natural isomorphism
$(N \otimes_A M) \otimes_U A \simeq N
 \otimes_\Uop M$.
\end{lemma}
\begin{proof}
Write first 
$A \otimes_\Uop (N \otimes_A M)$
rather than $(N \otimes_A M) \otimes_U A$,
and then apply 
the natural $k$-module isomorphism
$$
		  P \otimes_\Uop (N \otimes_A
		  M)\simeq 
		  (P \otimes_A N) \otimes_\Uop M
$$
from \cite[Lem.\ 3]{KowKra:DAPIACT} with
$P=A$. 
\end{proof}

Note that \cite[Lem.\ 3]{KowKra:DAPIACT}
applied with $P = A$, $M = \Aop$ 
yields the coinvariants in the
form used in \cite{KowPos:TCTOHA}
where they were considered 
as a functor $\umod \to \kmod$.

%In other words, both $B_\emme^{\bull +1}U=C^{\bull +1}_AU
%\otimes_\Ae M$
%and its quotient $C^{\bull +1}_AU
%\otimes_\Uop M$ can be canonically
%identified with the coinvariants of one
%and the same object 
%$C^{\bull +1}_AU
%\otimes_A M$, but with respect to
%different left Hopf algebroids, namely $\Ae$ and $U$, respectively.

% $U$ for the
% quotient and $V=\Ae$ for $B^\bull_\emmeU$.

%Next we would like to point out that the
%canonical quotient (of $k$-modules)
Applying Lemma~\ref{eli} with
$N=U^{\otimes_A \bull +1}$ 
will lead to the simpler form of the
para-cocyclic $k$-module we are going to
consider. To get there, we first remark:
% ,
%we have
%$$
%		  C^{\bull +1}_AU
%		  \otimes_A M \rightarrow 
%		  C^{\bull +1}_AU
%		  \otimes_\Uop M 
%$$
%splits canonically.
% To see this, 

\begin{lem}
Let $M \in \modu$ and $N, P \in
 \umod$. Then one has
$$
		  (un \otimes_A p) \otimes_\Uop m
		  =
		  (n \otimes_A u_-p) \otimes_\Uop mu_+
$$
for all $m \in M$, $n \in N$, and $p \in P$.
\end{lem}

\begin{proof}
One has
\begin{equation*}
\begin{array}{rll}
&		  (un \otimes_A p) \otimes_\Uop m \\
= & (u_{+(1)}n \otimes_A
		  u_{+(2)}u_- p) 
		  \otimes_\Uop m &\mbox{by \rmref{Sch1},} \\
=& u_+(n \otimes_A u_-p) \otimes_\Uop m
		  & \mbox{by the monoidal structure in $\umod$,} \\
=& (n \otimes_A u_-p) \otimes_\Uop mu_+.& 
\end{array}
\end{equation*}
The well-definedness of the first
 operation follows from
 \rmref{Sch5} using \rmref{brot} and
 \rmref{salz}. 
\end{proof}

Using this we now obtain:

\begin{prop}
\label{excellentesoiree}
For $M \in \modu$ and $N \in \umod$,
 there is a canonical isomorphism of
 $k$-modules 
\begin{equation}
\label{doppellaeufer}
\phi: (U \otimes_A N) \otimes_\Uop M
		  \stackrel{\simeq}{\longrightarrow} N
		  \otimes_A M, 
\end{equation}
given by 
\begin{equation}\label{allesvielsimpleralsgedacht}
		  (u \otimes_A n) \otimes_\Uop m
			\mapsto 
		  u_-n \otimes_A mu_+.
\end{equation}
\end{prop}
\begin{proof}
The map $n \otimes_A m \mapsto (1
 \otimes_A n) \otimes_\Uop m$  
is obviously a right inverse to
 \rmref{allesvielsimpleralsgedacht}, and
 by the preceding lemma it is also a
 left inverse. 
\end{proof}

In particular, this yields an isomorphism
\begin{equation}\label{corelli}
		  \phi : U^{\otimes_A \bull+1} 
		  \otimes_\Uop M \rightarrow 	 
		  U^{\otimes_A \bull} \otimes_A M  =: C^\bull(U,M), 
\end{equation} 
and the latter will be the ultimate object of study. 

%and a quotient map
%$$
%		  \bar\phi: B^\bull_\emme U
%		  \rightarrow C^\bull_\emme U,\quad
%		  u^0 \otimes_A \cdots \otimes_A u^n
%		  \otimes_\Ae m \mapsto 
%		  u^0_{-(1)}u^1 \otimes_A \cdots \otimes_A u^0_{-(n)}u^n \otimes_A mu^0_+,
%$$
%which can be seen to be well-defined using the Takeuchi property \rmref{Sch38} and \rmref{Sch5}.
%This map has a canonical $k$-linear splitting
%$$
%		  \iota : C^\bull_\emme U \rightarrow B^\bull_\emme U
%$$
%obtained as the composition of
%$$
%		  C^\bull_\emme U \rightarrow
%		  C^{\bull +1}_\emme U \otimes_A
%		  M,\quad
%		  w \otimes_A m \mapsto (1 \otimes
%		 _A w) \otimes_A m
%$$
%with the canonical projection
%$C^{\bull +1}_\emme U \otimes_A M
%\rightarrow B^\bull_\emme U$.

\subsection{Cyclic cohomology with
  coefficients for left Hopf algebroids} 
Now we ask whether the para-cocyclic
structure of $\boehmcomplex^\bull(U,M)$ 
descends to $C^\bull(U,M)$.
%
%and this is the quotient to which we
%would like the para-cocyclic structure
%of $B^\bull_\emme U$ to descend. 
%Let now $M$ be a left $U$-comodule and
%simultaneously a right $U$-module where the respective $(A,A)$-bimodule structures are supposed to be compatible as in \rmref{campanilla},
%and where $U$ is a left Hopf algebroid as
%before.
%
%The tensor product $C^{\bull+1}_A U \otimes_A M$, 
%where the left $U$-action on 
%$C^{\bull+1}_A U$ is the diagonal action via the coproduct, is a right $U$-module to which we can apply the functor of coinvariants.
%This amounts to the isomorphism 
%given in this case explicitly by
%
%
%On the other hand, the $k$-module $B^n_\emme U$ is isomorphic to a quotient of $C^{\bull+1}_A U \otimes_A M$ by the $k$-module
%$$
%I^n = {\rm span}\{u^0 \otimes_A \cdots \otimes_A u^n 
%		  \otimes_A ma - 
%		  a \lact u^0 \otimes_A \cdots
%		  \otimes_A u^n \otimes_A m \mid a \in A \}. 
%$$
%Note that from \rmref{Sch5} and \rmref{campanilla} follows $I^n \subset \ker \phi$, such that we have a diagram
%\[
%\xymatrix{C^n_\emme U \ar[rr]^{\phi^{-1}}&&C^{n+1}_A U \otimes_A M \ar[dl]^{\pi}\\
%&B^n_\emme U, \ar[ul]^{\bar{\phi}}
%}
%\] 
%where $\pi$ denotes the canonical projection and $\bar{\phi}$ the induced map.
This is answered by a left
Hopf algebroid left comodule and right
module version of \cite[Prop.\
2.19]{BoeSte:CCOBAAVC}, which generalises 
Proposition 5.2.1 in
\cite{Kow:HAATCT}:

\begin{prop}
\label{fratellanza}
If $M$ is an anti Yetter-Drinfel'd
 module as in Definition \ref{SAYD},  
the operators  $(\gd'_\bull, \gs'_\bull,
 \tau'_\bull)$ on $\boehmcomplex^\bull(U,M)$ from
 \rmref{duracell}  
descend   
to well-defined operators on
 $U^{\otimes_A \bull+1} \otimes_\Uop M$. 
\end{prop}

\begin{proof}
One needs to prove that the operators
 $(\gd'_\bull, \gs'_\bull, \tau'_\bull)$
 are $\Uop$-balanced,  
i.e., that one has for example 
$$
		  \tau'_n(u^0 \otimes_\ahha \cdots
		  \otimes_\ahha u^n \otimes_\Uop mv)  
		  = 
		  \tau'_n(v_{(1)} u^0
		   \otimes_\ahha \cdots \otimes_\ahha
		  v_{(n+1)} u^n \otimes_\Uop m) 
$$
for any $v \in U$. This is shown by
 expressing the right hand side as 
\begin{equation*}
\begin{split}
v_{(2)} & u^1 \otimes_\ahha \cdots \otimes_\ahha v_{(n+1)} u^n \otimes_\ahha m_{(-1)} v_{(1)} u^0 \otimes_\Uop m_{(0)} \\
&= v_{(2)} u^1 \otimes_\ahha \cdots \otimes_\ahha v_{(n+1)} u^n \otimes_\ahha \sll\eps(v_{(n+2)}) m_{(-1)} v_{(1)} u^0 \otimes_\Uop m_{(0)} \\
&= v_{(2)} u^1 \otimes_\ahha \cdots \otimes_\ahha v_{(n+1)} u^n \otimes_\ahha v_{{(n+2)}+} v_{{(n+2)}-} m_{(-1)} v_{(1)} u^0 \otimes_\Uop m_{(0)} \\
&= v_{+(2)} u^1 \otimes_\ahha \cdots \otimes_\ahha v_{+(n+1)} u^n \otimes_\ahha v_{+(n+2)} v_{-} m_{(-1)} v_{+(1)} u^0 \otimes_\Uop m_{(0)} \\ 
&= u^1 \otimes_\ahha \cdots \otimes_\ahha u^n \otimes_\ahha v_{-} m_{(-1)} v_{+(1)} u^0 \otimes_\Uop m_{(0)}v_{+(2)} \\ 
&= u^1 \otimes_\ahha \cdots
 \otimes_\ahha u^n \otimes_\ahha (mv)_{(-1)}
 u^0 \otimes_\Uop (mv)_{(0)},   
\end{split}
\end{equation*}
which is the left hand side.
Here we used the counital identities of
 the left coproduct in the second line,
 \rmref{Sch47} in the third line,
 \rmref{Sch38} combined with (higher)
 coassociativity in the fourth line, and
finally the anti Yetter-Drinfel'd condition
 \rmref{huhomezone}.  
Similar calculations can be made for the
 cofaces and codegeneracies.  
\end{proof}

We denote the resulting para-cocyclic
structure on 
$C^\bull(U,M)$ by 
\begin{equation}\begin{array}{rcl}\label{figc}
		  \gd_i 
&:=& 
		  \phi \circ  \bar\gd'_i \circ \phi ^{-1}, \\
		  \gs_i 
&:=& \phi \circ \bar\gs'_i \circ \phi ^{-1}, \\
		  \tau_i 
&:=& \phi \circ \bar\tau'_i \circ \phi ^{-1},
\end{array}
\end{equation}
where $\phi$ is the map from
\rmref{corelli} and $\bar \delta_i',\bar
\sigma'_j,\bar \tau'_n$ 
are the para-cocyclic operators on
$U^{\otimes_A \bull+1} \otimes_\Uop M$
that descend from $\boehmcomplex^\bull(U,M)$.

A short computation yields the explicit expressions 
given in Theorem~\ref{main} in the introduction:
\begin{equation}
\!\! \begin{array}{rll}
\label{anightinpyongyang}
\gd_i(z \otimes_\ahha m) \!\!\!\!&= \left\{\!\!\!
\begin{array}{l} 1%_\uhhu 
\otimes_\ahha u^1 \otimes_\ahha \cdots
 \otimes_\ahha u^n \otimes_\ahha m  
\\ 
u^1 \otimes_\ahha \cdots \otimes_\ahha \Deltaell (u^i) \otimes_\ahha \cdots
 \otimes_\ahha u^n \otimes_\ahha m
\\
u^1 \otimes_\ahha \cdots \otimes_\ahha u^n \otimes_\ahha m_{(-1)} \otimes_\ahha m_{(0)} 
\end{array}\right. \!\!\!\!\!\!\!\! 
& \!\! \hspace*{-2cm}  \begin{array}{l} \mbox{if} \ i=0, \\ \mbox{if} \
  1 \leq i \leq n, \\ \mbox{if} \ i = n + 1,  \end{array} \\
\gd_j(m) \!\!\!\! &= \left\{ \!\!\!
\begin{array}{l}
		  1%_\uhhu 
		  \otimes_\ahha m  \quad
\\
m_{(-1)} \otimes_\ahha m_{(0)}  \quad 
\end{array}\right. & \!\! \hspace*{-2cm} 
\begin{array}{l} \mbox{if} \ j=0, \\ \mbox{if} \
  j = 1 ,  \end{array} \\
\gs_i(z \otimes_\ahha m) \!\!\!\! 
&= u^1 \otimes_\ahha \cdots \otimes_\ahha
\eps (u^{i+1}) \otimes_\ahha \cdots \otimes_\ahha u^n \otimes_\ahha m & \! \,
\hspace*{1pt} \hspace*{-2cm}  0 \leq i \leq n-1,
\\
\tau_n(z \otimes_\ahha m) \!\!\!\! 
&= u^1_{-(1)}u^2 \otimes_\ahha \cdots \otimes_\ahha u^1_{-(n-1)}u^n\otimes_\ahha u^1_{-(n)}m_{(-1)} \otimes_\ahha m_{(0)}u^1_+, & 
\end{array}
\end{equation}
where we abbreviate $z:=u^1 \otimes_\ahha
\cdots \otimes_\ahha u^n$.

In this form, the well-definedness and the well-definedness over the Sweedler presentations of these
operators can be seen directly 
(using \rmref{Sch5} as well as the
Takeuchi properties of $\Deltaell$ and
$\gD_\emme$). 
Observe, however, that the condition $ma = m\sll(a)$ from \rmref{campanilla} 
is {\em not} needed to make 
the operators \rmref{anightinpyongyang} well-defined and
well-defined over the Sweedler
presentation but only to give a sense to
the above quotienting process.

It is less obvious
that the stability condition on $M$
implies cyclicity. This is, however,
immediate from the presentation of
$C^\bull(U,M)$ as a quotient of
$\boehmcomplex^\bull(U,M)$:

\begin{thm}
\label{bufalinadachteich}
If $U$ is a left Hopf algebroid and 
$M$ is a stable anti Yetter-Drinfel'd module, 
then $(C^\bull(U,M), \gd_\bull,
 \gs_\bull, \tau_\bull)$ is a cocyclic 
$k$-module.
\end{thm}

\begin{proof} 
By its construction, 
$(C^\bull(U,M), \gd_\bull,
 \gs_\bull, \tau_\bull)$ is a
 para-cocyclic object and as such
isomorphic to
$(U^{\otimes_A \bull+1} \otimes_\Uop
 M,\bar \gd'_\bull,\bar
 \gs'_\bull,\bar \tau'_\bull)$
obtained in
Proposition \ref{fratellanza}.
It remains to show that this quotient of 
$\boehmcomplex^\bull(U,M)$ is cocyclic if $M$ is stable:
\begin{equation*}
\begin{split}
(\bar\tau'_n)^{n+1}(u^0 \otimes_\ahha \cdots \otimes_\ahha u^n \otimes_\Uop m) 
&= m_{(-n-1)} u^0 \otimes_\ahha \cdots \otimes_\ahha m_{(-1)} u^n \otimes_\Uop m_{(0)} \\
&= m_{(-1)} \cdot \big(u^0 \otimes_\ahha \cdots \otimes_\ahha u^n \big) \otimes_\Uop m_{(0)} \\
&= u^0 \otimes_\ahha \cdots \otimes_\ahha u^n \otimes_\Uop m_{(0)}m_{(-1)}, \\
\end{split}
\end{equation*} 
where $\cdot$ denotes the
 diagonal left $U$-action via the
 left coproduct.  
\end{proof}

By the last line in the 
proof of the preceding theorem one may be tempted to think that an aYD module defines a para-cocyclic module which is cocyclic if $M$ is stable. The observation we add here is 
that for defining a para-cocyclic module the aYD property \rmref{huhomezone}, i.e.\ compatibility between $U$-action and $U$-coaction, 
is {\em not} required:

\begin{thm}
\label{jetztjetitlos}
Let $U$ be a left Hopf algebroid and $M$ a right $U$-module and left $U$-comodule, and let the respective left $A$-actions be compatible in the following sense:
\begin{equation}
\label{ziehteuchwarman}
am = a \blact m, \qquad m \in M, \ a \in A.
\end{equation}
Then $(C^\bull(U,M), \gd_\bull,
 \gs_\bull, \tau_\bull)$ is a para-cocyclic 
$k$-module.
\end{thm}
\begin{proof}
We need to check the relations in the right column in \rmref{belleville}. 
Since we do not assume that $M$ is aYD here, i.e.\ compatibility between action and coaction, 
the only relations that need to be checked 
are those that have the $U$-action on $M$ followed 
by an operation involving the $U$-coaction on $M$. 
Here, this is only $\tau_n \circ \gs_0 = \gs_n \circ \tau^2_{n+1}$, which is proven as follows: first compute 
\begin{equation*}
\begin{split}
\gs_n & \big(\tau_{n+1}(u^1 \otimes_\ahha \cdots \otimes_\ahha u^{n+1} \otimes_\ahha m)\big) \\
&= 
\gs_n\big(u^1_{-(1)}u^2 \otimes_\ahha \cdots \otimes_\ahha u^1_{-(n)} u^{n+1}\otimes_\ahha u^1_{-(n+1)}m_{(-1)} \otimes_\ahha m_{(0)}u^1_+\big) \\
&= 
u^1_{-(1)}u^2 \otimes_\ahha \cdots \otimes_\ahha t\big(\eps(u^1_{-(n+1)} m_{(-1)})\big) u^1_{-(n)} u^{n+1} \otimes_\ahha m_{(0)}u^1_+ \\
&= 
u^1_{-(1)}u^2 \otimes_\ahha \cdots \otimes_\ahha u^1_{-(n)} u^{n+1} \otimes_\ahha m_{(0)}\tl(\eps(m_{(-1)}))u^1_+ \\
&= 
u^1_{-(1)}u^2 \otimes_\ahha \cdots \otimes_\ahha u^1_{-(n)} u^{n+1} \otimes_\ahha m u^1_+, 
\end{split}
\end{equation*}
where we used the Takeuchi property \rmref{Sch3} in the fourth line and \rmref{ziehteuchwarman} together with
the comodule properties in the fifth, so that terms involving the coaction disappear.
Hence
\begin{equation*}
\begin{split}
& \gs_n \tau^2_{n+1}(u^1 \otimes_\ahha \cdots \otimes_\ahha u^{n+1} \otimes_\ahha m) \\
&= 
\gs_n\tau_{n+1}\big(u^1_{-(1)}u^2 \otimes_\ahha \cdots \otimes_\ahha u^1_{-(n)} u^{n+1}\otimes_\ahha u^1_{-(n+1)}m_{(-1)} \otimes_\ahha m_{(0)}u^1_+\big) \\
&=
(u^1_{-(1)}u^2)_{-(1)} u^1_{-(2)}u^3 \otimes_\ahha \cdots 
\otimes_\ahha (u^1_{-(1)}u^2)_{-(n)} u^1_{-(n+1)}m_{(-1)} \otimes_\ahha m_{(0)}u^1_+(u^1_{-(1)}u^2)_+  \\
&= u^2_{-(1)}\big((u^1_{-})_{(1)-} (u^1_-)_{(2)}\big)_{(1)} 
u^3 \otimes_\ahha \cdots \\
& \hspace*{2cm}
\otimes_\ahha u^2_{-(n)} \big((u^1_{-})_{(1)-} (u^1_{-})_{(2)}\big)_{(n)}  m_{(-1)} \otimes_\ahha m_{(0)}u^1_+ (u^1_{-})_{(1)+}u^2_+  \\
&= u^2_{-(1)} u^3 \otimes_\ahha \cdots 
\otimes_\ahha u^2_{(n-1)} u^{n+1} \otimes_\ahha u^2_{-(n)} m_{(-1)} \otimes_\ahha m_{(0)}\sll(\eps(u^1)) u^2_+, 
\end{split}
\end{equation*}
where in the fifth line \rmref{Sch2} was used and \rmref{Sch47} in the sixth. By \rmref{Sch5} this is now
easily seen to be equal to $\tau_n\gs_0(u^1 \otimes_\ahha \cdots \otimes_\ahha u^{n+1} \otimes_\ahha m)$.
\end{proof}

\begin{dfn}
For a right $U$-module left $U$-comodule $M$ with compatible induced left $A$-actions
over a 
left Hopf algebroid $U$, 
we denote by 
$H^\bull(U, M)$ and $HC^\bull(U,M)$ the
simplicial and cyclic cohomology groups
of $C^\bull(U,M)$. We refer to $HC^\bull(U,M)$ as to
the 
{\em  Hopf-cyclic cohomology of $U$ with
 coefficients in $M$}.
\end{dfn}

Note that the simplicial cohomology is the
ordinary $\mathrm{Cotor}$ over $U$:

\begin{prop}\label{daumenschraube}\cite{Kow:HAATCT,
 KowPos:TCTOHA} 
If $\due U {} \ract$ is flat as right
 $A$-module, then one has
$$
H^\bull(U,M) \simeq \Cotor^\bull_U(A,M).
$$
\end{prop}

\begin{rem}
\label{citoyen}
If $U$ is a (full) Hopf algebroid over
 base algebras $A$ and $B \simeq
 \Aop$,  
it is easy to check that $B$ fulfills
 the properties of an anti
 Yetter-Drinfel'd module with respect to
 the right $U$-action given by the right
 counit of the underlying right
 bialgebroid. This module is stable
if
the antipode
 of the Hopf algebroid is an
 involution. The operators
 \rmref{anightinpyongyang} reduce here to the
 well-known Hopf-cyclic operators for
 Hopf algebroids, cf.\
 \cite{ConMos:DCCAHASITG,
 KhaRan:PHAATCC, Kow:HAATCT,
 KowPos:TCTOHA}. For example, the cyclic
 operator reduces in such a case to 
$$
\tau_n(h^1 \otimes_\ahha \cdots
 \otimes_\ahha h^n) = (S(h^1))_{(1)} h^2
 \otimes_\ahha \cdots \otimes_\ahha
 (S(h^1))_{(n-1)}  
h^n\otimes_\ahha (S(h^1))_{(n)}.
$$
\end{rem}

\section{Hopf-Cyclic Homology with Coefficients}
\label{wirdschon}

\subsection{Cyclic homology with coefficients for left Hopf algebroids}

Let $U$ be a left Hopf algebroid 
%($\times_A$-Hopf algebra) 
over $A$ with structure maps as before,
and let $M$ be a left $U$-comodule 
with left coaction denoted 
$\gD_\emme : m \mapsto m_{(-1)} \otimes_A m_{(0)}$ with
underlying left $A$-action $(a, m)
\mapsto am$,  
and simultaneously a right $U$-module
with right action denoted $(m, u)
\mapsto mu$, subject to the compatibility condition
%$$
%am = a \blact m, \quad a \in A, \ m \in M.
%$$
\rmref{campanilla} with respect the two
induced $\Ae$-module structures.

Now define
$$
		  C_\bull(U,M) := M \otimes_\Aopp  
		  (\due U \blact \ract)^{\otimes_\Aopp  \bull},
$$ 
where the tensor product is formed as in
\rmref{tata}.
On $C_\bull(U,M)$, define the following
operators, abbreviating $x:=u^1
\otimes_\Aopp \cdots \otimes_\Aopp u^n$: 
\begin{equation}
\!\!\!
\begin{array}{rcll}
\label{adualnightinpyongyang}
d_i(m \otimes_\Aopp x)  &\!\!\!\!\! =& \!\!\!\!\!
\left\{ \!\!\!
\begin{array}{l}
m \otimes_\Aopp u^1  \otimes_\Aopp   \cdots   \otimes_\Aopp   \eps(u^n) \blact u^{n-1}
\\
m \otimes_\Aopp \cdots \otimes_\Aopp  u^{n-i} u^{n-i+1}
 \otimes_\Aopp  \cdots 
\\
mu^1 \otimes_\Aopp u^2  \otimes_\Aopp   \cdots    \otimes_\Aopp  
u^n 
\end{array}\right.  & \!\!\!\!\!\!\!\!\!\!\!\! \,  \begin{array}{l} \mbox{if} \ i \!=\! 0, \\ \mbox{if} \ 1
\!  \leq \! i \!\leq\! n-1, \\ \mbox{if} \ i \! = \! n, \end{array} \\
s_i(m \otimes_\Aopp x) &\!\!\!\!\! =&\!\!\!\!\!  \left\{ \!\!\!
\begin{array}{l} m  \otimes_\Aopp   u^1  
\otimes_\Aopp   \cdots   \otimes_\Aopp
 u^n  \otimes_\Aopp 
		  1%_\uhhu
\\
m \otimes_\Aopp \cdots \otimes_\Aopp   u^{n-i} 
\otimes_\Aopp   1%_\uhhu  
\otimes_\Aopp   u^{n-i+1}  \otimes_\Aopp
 \cdots  
\\
m \otimes_\Aopp 1%_\uhhu 
\otimes_\Aopp u^1  \otimes_\Aopp   \cdots    \otimes_\Aopp  u^n 
\end{array}\right.   & \!\!\!\!\!\!\!\!\!\!\!  \begin{array}{l} 
\mbox{if} \ i\!=\!0, \\ 
\mbox{if} \ 1 \!\leq\! i \!\leq\! n-1, \\  \mbox{if} \ i\! = \!n, \end{array} \\
t_n(m \otimes_\Aopp x) 
&\!\!\!\!\!=&\!\!\!\!\! 
m_{(0)} u^1_+ \otimes_\Aopp u^2_+ \otimes_\Aopp  \cdots  \otimes_\Aopp u^n_+ \otimes_\Aopp u^n_- \cdots u^1_- m_{(-1)}.  
& \\
\end{array}
\end{equation}
Elements of degree zero (i.e.\ of $M$)
are mapped to zero by the face maps, 
$d_0(m) = 0$ for all $m \in
M$. Well-definedness and
well-definedness over the various
Sweedler presentations follows from
\rmref{Sch3}, \rmref{Sch5},
\rmref{douceuretresistance}, and
\rmref{maotsetung}. 
Similarly as in the cohomology case, 
these operators still make sense if one drops the condition $ma = ms(a)$ from the axiom \rmref{campanilla} as well as the aYD condition 
\rmref{huhomezone}. 

As one might expect, we will obtain dually to Theorems~\ref{jetztjetitlos}~\&~\ref{bufalinadachteich}: 

\begin{thm}
\label{trocadero}
Let $U$ be a left Hopf algebroid.
\begin{enumerate}
\item
If $M$ is a right $U$-module and left $U$-comodule with respective left $A$-actions compatible as in \rmref{ziehteuchwarman},
then $(C_\bull(U,M), d_\bull,
 s_\bull, t_\bull)$ is a para-cyclic 
$k$-module.
\item
If $M$ is even a stable anti Yetter-Drinfel'd module,
then $(C_\bull(U,M), d_\bull, s_\bull,
 t_\bull)$ is a cyclic  
$k$-module.
\end{enumerate}
\end{thm}

We will prove this below 
by presenting $C_\bull(U,M)$ as 
a cyclic dual of $C^\bull(U,M)$.

% \noindent Dually to the considerations in Section \ref{effacersvp}, the observation here again is that the 
% aYD condition \rmref{huhomezone} on $M$ is not needed to define a para-cyclic module:
%
% \begin{thm}
% \label{jetztjetitlos2}
% Let $U$ be a left Hopf algebroid and $M$ a right $U$-module and left $U$-comodule, where the respective left $A$-actions be compatible as % 
% in \rmref{ziehteuchwarman}.
% Then $(C_\bull(U,M), d_\bull,
% s_\bull, t_\bull)$ is a para-cyclic 
% $k$-module.
% \end{thm}
% \begin{proof}
% Analogously to the proof of Theorem \ref{jetztjetitlos}, we have to show the relations in the left column of \rmref{belleville}. Again, the 
% only critical ones are those where the $U$-coaction on $M$ follows an operation involving the $U$-action on $M$. 
% Here, this is only $s_0 \circ t_n = t^2_{n+1} \circ s_n$,
% which is obvious by coassociativity  of the comodule, 
% and since the involved $U$-action in this case is by the element $1 \in U$, due to the specific form of $s_n$.
% \end{proof}

\begin{dfn}
For a right $U$-module left $U$-comodule $M$ with compatible induced left $A$-actions
over a 
left Hopf algebroid $U$, 
we denote by
$H_\bull(U, M)$ and $HC_\bull(U,M)$ the
simplicial and cyclic homology groups of 
$C_\bull(U,M)$. We refer to 
$HC_\bull(U,M)$ as to the
{\em Hopf-cyclic homology of $U$ with
 coefficients in $M$}.
\end{dfn}

Dually to
Proposition~\ref{daumenschraube}, one
has:

\begin{prop}\cite{Kow:HAATCT, KowPos:TCTOHA}
\label{obecni}
If $\due U
 \blact {}$ is
 projective as left $A$-module, then one has 
$$
		  H_\bull(U,M)\simeq \Tor^U_\bull(M,A).
$$
\end{prop}
 
\begin{rem}
As in Remark \ref{citoyen}, in a full Hopf algebroid $H$ the base algebra $B$ of the underlying right bialgebroid 
is an anti
Yetter-Drinfel'd module which is stable if the antipode is an involution. 
The cyclic operator assumes the form
$$
t_n(u_1  \otimes_\Aopp   \cdots   \otimes_\Aopp u_n) = u_2^{(1)} \otimes_\Aopp \cdots \otimes_\Aopp u_{n}^{(1)} \otimes_\Aopp S(u_1 u_{2}^{(2)} \cdots u_{n}^{(2)}),
$$
where the Sweedler superscripts refer to
 the {\em right} coproduct. 
This is the same expression as the {\em inverse} of the cyclic operator given in \cite{Kow:HAATCT, KowPos:TCTOHA}, see our explanations below.
\end{rem}

%In this section we shall prove that
%$C^\emme_\bull U$ is
%isomorphic to $C_\emme^\bull U$ as a $k$-module 
%and that
%via this isomorphism, $(C_\bull^\emme U,
%d_\bull, s_\bull, t_\bull)$ is the
%cyclic dual to $(C^\bull_\emme U,
%\gd_\bull, \gs_\bull, \tau_\bull)$, a
%notion we will briefly recall, but in a
%slightly modified version.  
%This in particular proves that $C^\emme_\bull U$ is indeed a cyclic
%module.

\subsection{Cyclic duality}
{\cite{Con:CCEFE,Elm:ASFFCD,FeiTsy:AKT,Lod:CH}}  
Recall that the cyclic category
is self-dual, that is, we have $\Lambda_1 \simeq \Lambda_1^\op$, and therefore
cocyclic $k$-modules and cyclic
$k$-modules can be canonically identified.
However, there are even infinitely many
such canonical identifications since 
the cyclic category has many autoequivalences (see
e.g.~\cite[6.1.14 \& E.6.1.5]{Lod:CH}, 
but
note that the very last line of \cite[6.1.14]{Lod:CH}
should read $\tau_n \mapsto
\tau_n^{-1}$).

Fe{\u\i}gin and Tsygan have generalised the
duality to their category
$\Lambda_\infty$, that is, to para-(co)cyclic $k$-modules whose
cyclic operators are isomorphisms
(see \cite{FeiTsy:AKT}, Section A7).
Unfortunately, they use the most common
choice of equivalence $\Lambda_\infty
\simeq \Lambda_\infty^\op$ which does
not extend to
general para-(co)cyclic objects.

However, a different equivalence $\Lambda_\infty \simeq
\Lambda_\infty^\op$ \emph{does} lift to
a functor
$\Lambda^\op \rightarrow
\Lambda$, so that one can 
assign a para-cyclic module to any
para-cocyclic module even with not
necessarily invertible $\tau_n$, one only
has to bear in mind that this process is in general not invertible.
Still, it can be applied in full
generality to the para-cocyclic
object $C^\bull(U,M)$, even when $M$
is not SaYD, and hence 
Theorem~\ref{trocadero} follows from the
results of the previous section.

Explicitly, we use the
following convention for this
functor. We decided to stick to the term
``cyclic dual'' although it is no longer
a true duality in general:
\begin{definition}
The {\em cyclic dual} of  
a para-cocyclic $k$-module
$C^\bull = (C^\bull, \gd_\bull,
\gs_\bull, \tau_\bull)$ is
% with invertible operator $\tau$. 
the cyclic $k$-module 
$\hatC_\bull := (\hatC_\bull, d_\bull, s_\bull,
t_\bull)$, where $\hatC_n := C^n$, and
\begin{equation}
\label{moulesfrites}
\begin{array}{lclcll}
%d_0 := \gs_{n-1} & \!\!: & \hatC_n & \to & \hatC_{n-1}, &  \\
d_i := \gs_{n-(i+1)} & \!\!: & \hatC_n & \to & \hatC_{n-1}, & \quad 0 \leq i < n, \\
d_n := \gs_{n-1} \circ \tau_n  & \!\!: & \hatC_n & \to & \hatC_{n-1}, &  \\
s_i := \gd_{n-(i+1)} & \!\!: & \hatC_{n-1} & \to & \hatC_{n}, & \quad 0 \leq i < n, \\
t_n :=  \tau_n  & \!\!: & \hatC_n & \to & \hatC_{n}. &
\end{array}
\end{equation}
\end{definition}

For the convenience of the reader we
verify at least some of the relations:
\begin{lemma}
The cyclic dual of any para-cocyclic
 $k$-module is a para-cyclic $k$-module.
\end{lemma}
\begin{proof}
We need to check the para-cyclic relations
by using the para-cocyclic ones,
which
is straightforward. For example,
let $i < j$ and $j < n$; then $n - (i+2) \geq n- (j+1)$, and
$$
d_i \circ d_j = \gs_{(n-1) - (i+1)} \circ \gs_{n - (j+1)} = \gs_{(n - 1) - j}\circ \gs_{n - (i+1)} = d_{j-1} \circ d_i. 
$$
For $j = n$ (in which case $i \leq n-2$),
\begin{equation*}
\begin{split}
d_i \circ d_n &= \gs_{(n-1) - (i+1)} \circ \sigma_{n-1} \circ \tau_n \\
&= \gs_{n-2} \circ \gs_{(n-1) - (i+1)} \circ \tau_n \\
&= \gs_{n-2} \circ \tau_{n-1} \circ \gs_{n - (i+1)} = d_{n-1} \circ d_i.
\end{split}
\end{equation*}
Likewise, 
$$
d_i \circ s_i = \gs_{n-(i+1)} \circ \gd_{n-(i+1)} = \id = \gs_{n - j - 2} \circ \gd_{n-j -1} = d_{j+1} \circ s_j.
$$
Also
$$
d_i \circ t_n = \gs_{n- (i+1)} \circ \tau_n = \gs_{n-i-1} \circ \tau_n = \tau_{n-1} \circ \gs_{n-i} = t_{n-1} \circ d_{i-1}
$$
for $1 \leq i \leq n-1$, and for $i=n$ the identity 
$
d_0 \circ t_n = d_n
$
is trivially fulfilled. Finally,
$$
s_0 \circ t_n = \gd_{n-1} \circ \tau_n = \tau_{n+1} \circ \gd_n = \tau_{n+1} \circ \tau_{n+1} \gd_0 = t^2_{n+1} \circ s_n. 
$$
The rest of the simplicial and cyclic identities are left to the reader.
\end{proof}

\begin{rem}
 Note that the last coface map $\gd_{n} : C^{n-1}
 \rightarrow C^n$ is not
used in the construction of the cyclic dual: there is one less degeneracy
 $s_i : C_{n-1} \rightarrow C_n$ than there are cofaces
 $\delta_i : C^{n-1}
 \rightarrow C^n$. Conversely, there are
not enough codegeneracies to derive all
the face maps:
the last face map $d_n$ uses the extra
codegeneracy $\gs_{n-1} \circ \tau_n$ that
 arises from the (para-)cocyclic operator.
\end{rem}    

\begin{rem}
Observe that the cyclic homology of the
cyclic dual of a given cocyclic
$k$-module is independent of the choice 
of the self-duality of the cyclic
category $\Lambda_1$. 
This follows from the
 description of cyclic homology as 
$\mathrm{Tor}^{\Lambda^\op_1}_\bull(k,C)$ 
(cf.~\cite{Lod:CH}, Theorem~6.2.8)
in combination with the fact
 that all autoequivalences of $\Lambda_1$
 leave the trivial cyclic $k$-module $k$
invariant.  
\end{rem}

\begin{rem}
Two relatively straightforward cases in which the cyclic
 operator is not
 invertible are that of a Hopf algebra $U$ 
(over $A=k$) whose
antipode is not bijective, taking the
 coefficients to be $M=k$ with trivial action
$1 \cdot
 u=\eps(u)$ and trivial coaction $\Delta_\emme (1)=1 \otimes
 1$; or that of
$U=A^e$, $M=A_\sigma$, discussed in Section~\ref{beispielabsch}
below, when $\sigma$ is not bijective.

However, it seems worthwhile to remark
 that $\tau$ is invertible if $U$ is a full
Hopf algebroid with invertible antipode $S$ and $M$ has yet some additional structure:
recall first \cite{Boe:GTFHA} that
the two constituting bialgebroids (i.e.\ left and right) in a full Hopf algebroid 
have different underlying corings
(over anti-isomorphic base algebras) that have a priori different categories of comodules. 
A {\em Hopf algebroid (say, left) comodule} is then, roughly speaking, both a left and right bialgebroid (left) 
comodule, the two structures being compatible with each other.
If $M$ is a left comodule over the full Hopf algebroid $U$ and aYD in the sense of Definition \ref{SAYD} with respect to the underlying left bialgebroid, 
one checks by a tedious but straightforward induction on $n$ that
$$
w \otimes_\ahha m \mapsto 
\big(S^{-2}(u^n_-) m^{(-1)}{}_{(1)}\big) \cdot \big(1 \otimes_\ahha u^1 \otimes_\ahha \cdots \otimes_\ahha u^{n-1}\big) \otimes_\ahha m^{(0)} 
S^{-1}(m^{(-1)}{}_{(2)})S^{-2}(u^n_+)  
$$
yields an inverse for the cocyclic operator $\tau_n$ from \rmref{anightinpyongyang}, where we abbreviated
$w:=u^1 \otimes_\ahha \cdots \otimes_\ahha u^n$.
Here $\cdot$ denotes the diagonal action via the left coproduct and Sweedler superscripts the left coaction with respect to the underlying right bialgebroid in $U$. In case $M=B \simeq \Aop$, this reduces to the well-known expression
$$
		  u^1 \otimes_\ahha \cdots
 \otimes_\ahha u^n \mapsto 
		  (S^{-1}(u^n)) \cdot \big(1 \otimes_\ahha u^1 \otimes_\ahha \cdots \otimes_\ahha u^{n-1}\big)  
$$
from \cite{Kow:HAATCT, KowPos:TCTOHA}.
If $M$ is an SaYD so that
$C^\bull(U,M)$ is cocyclic, then  
the inverse of $\tau_n$ is simply given for any
 left Hopf algebroid $U$ by
\begin{equation*}
		  \tau^{-1}_n 
		  (u^1 \otimes_\ahha \cdots \otimes_\ahha u^n \otimes_\ahha m) 
		  = 
		  u^n_{-}m_{(-1)} \cdot \big(1 \otimes_\ahha u^1 
		  \otimes_\ahha \cdots \otimes_\ahha u^{n-1}\big) \otimes_\ahha m_{(0)}u^n_+. 
\end{equation*} 
\end{rem}

\subsection{The Hopf-Galois map and cyclic duality}
\label{adartem}
The explicit map implementing the isomorphism 
 $C_\bull (U,M) \simeq C^\bull (U,M)$ is given by 
generalising the Hopf-Galois map \rmref{Galois}:
\begin{lemma}
\label{hopf-galois}
For each $n \geq 0$, the $k$-modules $C_n(U,M)$ and $C^n(U,M)$ are
isomorphic by means of the Hopf-Galois map
$\varphi_n:C_n( U,M) \to C^n(U,M)$ in degree $n$, defined 
by $\varphi_0:=\id_\emme$, $\varphi_1: m \otimes_\Aopp u \mapsto u \otimes_\ahha m$, 
and for $n \geq 2$
\begin{equation}
\label{Hin}
\varphi_n: m \otimes_\Aopp u^1  \otimes_\Aopp  \cdots \otimes_\Aopp u^n \mapsto 
u^1_{(1)} \otimes_\ahha u^1_{(2)} u^2_{(1)} \otimes_\ahha \cdots \otimes_\ahha u^1_{(n)} u^2_{(n-1)} \cdots u^{n-1}_{(2)}u^n \otimes_\ahha m,
\end{equation}
with inverse
\begin{equation*}
\label{Her}
\psi_n: u^1 \otimes_\ahha \cdots \otimes_\ahha u^n \otimes_\ahha m \mapsto  
m \otimes_\Aopp u^1_+ \otimes_\Aopp u^1_-  u^2_+  \otimes_\Aopp u^2_- u^3_+  \otimes_\Aopp \cdots  \otimes_\Aopp u^{n-1}_-  u^n.
\end{equation*}
\end{lemma}
\begin{proof}
Well-definedness and well-definedness over the respective Sweedler
presentations follows from the Takeuchi conditions
\rmref{taki} and \rmref{Sch3}.
The fact that $\varphi$ and $\psi$ are mutually inverse is directly
checked by induction on $n$ using the properties \rmref{Sch1} and \rmref{Sch2}. 
\end{proof}

\begin{lem}
Let $U$ be a left Hopf algebroid 
%($\times_A$-Hopf algebra) 
with structure maps as before.
The Hopf-Galois map identifies $C_\bull (U,M)$ as the cyclic dual of the cocyclic module $C^\bull(U,M)$ of Theorem \ref{bufalinadachteich}.
\end{lem}
\begin{proof}
We need to show e.g.\ for the cyclic operators \rmref{anightinpyongyang} and \rmref{adualnightinpyongyang} 
$$
\tau_n \circ \varphi_n = \varphi_n \circ t_n
$$
with respect to the map \rmref{Hin}. 
This is a straightforward verification: one has
\begin{equation*}
\begin{split}
& \tau_n \varphi_n(m \otimes_\Aopp u^1 \otimes_\Aopp \cdots \otimes_\Aopp u^n) \\
&= 
\tau_n(u^1_{(1)} \otimes_\ahha u^1_{(2)} u^2_{(1)} \otimes_\ahha \cdots \otimes_\ahha u^1_{(n)} u^2_{(n-1)} \cdots u^{n-1}_{(2)}u^n \otimes_\ahha m) \\
&= 
u^1_{(1)-(1)} u^1_{(2)} u^2_{(1)} \otimes_\ahha \cdots \otimes_\ahha u^1_{(1)-(n-1)} u^1_{(n)} u^2_{(n-1)} \cdots u^{n-1}_{(2)}u^n \\
& \hspace*{6cm}
\otimes_\ahha u^1_{(1)-(n)} m_{(-1)} \otimes_\ahha m_{(0)} u^1_{(1)+}  \\
&= 
u^1_{(1)-(1)} u^1_{(2)} u^2_{(1)} \otimes_\ahha \cdots \otimes_\ahha u^1_{(1)-(n-1)} u^1_{(n)} u^2_{(n-1)} \cdots u^{n-1}_{(2)}u^n \\
& \hspace*{6cm}
\otimes_\ahha u^1_{(1)+-} m_{(-1)} \otimes_\ahha m_{(0)} u^1_{(1)++}  \\
&= 
u^2_{(1)} \otimes_\ahha u^2_{(2)} u^3_{(1)} \otimes_\ahha \cdots \otimes_\ahha u^2_{(n-1)} \cdots u^{n-1}_{(2)}u^n \otimes_\ahha u^1_{-} m_{(-1)} \otimes_\ahha m_{(0)} u^1_{+}  \\
\end{split}
\end{equation*}
using \rmref{Sch37} and \rmref{Sch2}; whereas
\begin{equation*}
\begin{split}
&\varphi_n t_n(m \otimes_\Aopp u^1 \otimes_\Aopp \cdots \otimes_\Aopp u^n) \\
&= \varphi_n (m_{(0)} u^1_+ \otimes_\Aopp u^2_+ \otimes_\Aopp  \cdots  \otimes_\Aopp u^n_+ \otimes_\Aopp u^n_- \cdots u^1_- m_{(-1)}) \\
&= u^{2}_{+(1)} \otimes_\ahha u^{2}_{+(2)}u^{3}_{+(1)} \otimes_\ahha \cdots \otimes_\ahha 
u^{2}_{+(n)}u^{3}_{+(n-1)} \cdots u^{n}_{+(2)} u^n_- \cdots u^1_- m_{(-1)}  \otimes_\ahha m_{(0)} u^1_+ \\
&= u^{2}_{(1)} \otimes_\ahha u^{2}_{(2)}u^{3}_{(1)} \otimes_\ahha \cdots \otimes_\ahha u^2_{(n-1)} \cdots u^{n-1}_{(2)} u^n \otimes_\ahha
u^1_- m_{(-1)}  \otimes_\ahha m_{(0)} u^1_+
\end{split}
\end{equation*}
by \rmref{Sch38} and \rmref{Sch1}, and the claim follows.
The corresponding identities relating (co)faces to (co)degeneracies are left to the reader.
\end{proof}

% \begin{cor}
% If $U$ is a left Hopf algebroid and $M$ a stable anti Yetter-Drinfel'd module, 
% $(C_\bull (U,M), d_\bull, s_\bull, t_\bull)$ forms a cyclic module.
% \end{cor}

\noindent {\it Proof}
({\it of Theorem \ref{trocadero}}). 
%and Theorem \ref{jetztjetitlos2} 
\ This now follows from Theorem \ref{bufalinadachteich} and Theorem \ref{jetztjetitlos}.
{\mbox{}\hfill$\bx$\vspace{1.5ex} \par}

\section{Examples}
\subsection{Lie-Rinehart homology with coefficients}

Let $(A,L)$ be a Lie-Rinehart algebra over a commutative $k$-algebra $A$ and $VL$ be its universal enveloping algebra 
(see \cite{Rin:DFOGCA}).

The {\em left Hopf algebroid structure} of $VL$ has been described in \cite{KowKra:DAPIACT}; 
as therein, we denote by the same symbols elements $a \in A$ and $X \in L$ and the corresponding generators in $VL$.
The maps $s = t$ are equal to the canonical injection $A \to VL$. The coproduct and the counit are given by 
$$
\begin{array}{rclrcl}
\gD(X) &:=& X \otimes_A 1 + 1 \otimes_A X, & \qquad \eps(X) &:=& 0, \\
\gD(a) &:=& a \otimes_A 1, &                 \qquad \eps(a) &:=& a,
\end{array}
$$
whereas the inverse of the Hopf-Galois map is
$$
X_+ \otimes_\Aop X_- := X \otimes_\Aop 1 - 1 \otimes_\Aop X, \qquad a_+ \otimes_\Aop a_- := a \otimes_A 1.
$$ 
By universality, 
these maps can be extended to $VL$.

Recall from \cite{Hue:DFLRAATMC} that a {\em right $(A,L)$-module} $M$ 
is simultaneously a left $A$-module with action $(a,m) \mapsto am$ and 
a right $L$-module with action $(m,X) \mapsto mX$, subject to the compatibility conditions
\begin{equation*}
\begin{array}{rl}
\begin{array}{rcl}
(am)X &=& a(mX) - X(a)m, \\
 m(aX)&=& a(mX) - X(a)m, 
\end{array}
& m \in M, \ a \in A, \ X \in L.
\end{array}
\end{equation*}
Right $(A,L)$-module structures correspond to right $VL$-module structures and vice versa.
For a right $(A,L)$-module $M$ we define {\em Lie-Rinehart homology with 
coefficients in $M$} as
\begin{equation}
\label{chezjeanette}
H_\bull(L,M) := \Tor^{VL}_\bull(M,A).
\end{equation}

%Assume $L$ to be $A$-projective for the rest of this section.
Interestingly enough, every right $(A,L)$-module is an SaYD module with respect to the trivial coaction (cf. Remark~\ref{cannibale}):

\begin{lem}
\label{menilmontant}
Let $M$ be any right $(A,L)$-module and define on $M$ a left $VL$-coaction by
$
\gD_\emme: M \to VL \otimes_A M, \ m \mapsto 1 \otimes_A m.
$
Then $M$ is a stable anti Yetter-Drinfel'd module.
\end{lem}
\begin{proof}
Equipped with this coaction, $M$ is obviously stable, and also \rmref{campanilla} is immediate (observe that left and right $A$-action on $M$ coincide). 
Hence it remains to show \rmref{huhomezone}. 
With the left Hopf algebroid structure maps mentioned above, it is easy to see that on generators
$$
\gD_\emme(mX) = 1 \otimes_A mX = X_- X_{+(1)} \otimes_A mX_{+(2)} = X_- m_{(-1)} X_{+(1)} \otimes_A m_{(0)} X_{+(2)}
$$ 
holds for $X \in L$, 
and trivially on generators $a \in A$. For an element 
$u =  aX_{1} \cdots X_{p}$, 
%$u =  aX_{i_1} \cdots X_{i_p}$ 
%of a PBW basis, 
where $a \in A$, $X_i\in L$, one immediately obtains
\begin{equation*}
\begin{split}
\gD(mu) &= 1 \otimes_A mu'X_{p} \\
        &=  (X_{p})_- (X_{p})_{+(1)} \otimes_A mu'(X_{p})_{+(2)} \\
        &= (X_{p})_- (mu')_{(-1)} (X_{p})_{+(1)} \otimes_A (mu')_{(0)}(X_{p})_{+(2)}
\end{split}
\end{equation*}
for
$u' =  aX_{1} \cdots X_{p-1}$. 
%$u' =  aX_{i_1} \cdots X_{i_{p-1}}$. 
By induction on $p$ and \rmref{Sch4} 
one concludes $\gD(mu) = u_- m_{(-1)} u_{+(1)} \otimes_A m_{(0)}u_{+(2)}$, as desired.
\end{proof}

%Let $\textstyle\bigwedge^\bull_\ahha  L$ be the exterior algebra over $A$ and $M$ a right $(A,L)$-module. 
Recall that there is a canonical complex that computes $H_\bull(L,M)$ whenever $L$ is $A$-projective. 
This is given by the exterior algebra $\bigwedge_A^\bull L$
tensored over $A$ with $M$, with differential 
$\partial=\partial_n:M \otimes_\ahha  \textstyle\bigwedge^n_\ahha L \to M \otimes_\ahha  \textstyle\bigwedge^{n-1}_\ahha L$ defined by
\begin{equation*}
\begin{split}
\partial(&m \otimes_\ahha  X_1\wedge\cdots\wedge X_n)\\
:=&\sum_{i=1}^n(-1)^{i+1} 
mX_i \otimes_\ahha  X_1\wedge\cdots\wedge\hat{X}_i\wedge\cdots\wedge X_n\\
&+\sum_{i<j}(-1)^{i+j}m \otimes_\ahha  [X_i,X_j]\wedge X_1\wedge\cdots\wedge\hat{X}_i
\wedge\cdots\wedge\hat{X}_j\wedge\cdots\wedge X_n.
\end{split}
\end{equation*}

\noindent The following theorem generalises \cite[Thm.\ 3.13]{KowPos:TCTOHA} to more general coefficients.

\begin{theorem}
Let $(A,L)$ be a Lie-Rinehart algebra, where $L$ is $A$-projective,
and $M$ a right $(A,L)$-module which is $A$-flat, seen also as a left $VL$-comodule as in Lemma \ref{menilmontant}. 
The map
$$
\Xi: m \otimes_\ahha  X_1 \wedge \cdots \wedge X_n \mapsto
\frac{1}{n!}\sum_{\gs \in S_n} (-1)^\gs X_{\gs(1)} \otimes_\ahha
\cdots \otimes_\ahha  X_{\gs(n)} \otimes_\ahha m
$$
defines a morphism of
mixed complexes
$$
\left( M \otimes_\ahha  \textstyle\bigwedge^\bull_\ahha L,0,\partial\right)\to
\left(C^\bull(VL, M),b,B \right)
$$
which induces natural isomorphisms
\begin{equation*}
\begin{split}
H^\bull(V L, M)&\simeq M \otimes_\ahha \textstyle\bigwedge^\bull_\ahha L,\\
%HP^\bull(V L, M) &\simeq\underset{n \equiv \bull \, {\rm mod} \, 2}{\textstyle{\bigoplus}} \, H_{n}(L,M). 
HC^\bull(V L, M) &\simeq \ker \pl_\bull \oplus H_{\bull-2}(L,M) 
\oplus H_{\bull-4}(L,M) \oplus \cdots. 
\end{split}
\end{equation*}
\end{theorem}

\begin{proof}
The first part of the theorem and the first isomorphism follow immediately by the 
form of the cosimplicial operators in \rmref{anightinpyongyang} for a trivial coaction, combined with 
the analogous result for  $M=A$ from \cite{KowPos:TCTOHA} 
and the flatness assumption on $M$.

To prove the second isomorphism, 
we need to show that $\Xi$ intertwines
the horizontal differential $B$ with 
$\partial$. 
This will be done by explicitly applying
 the coinvariants functor and the
 results in Section \ref{effacersvp}.  
Let $\tilde B: B^\bull(V L, M) \to
 B^{\bull-1}(V L, M)$ 
%and
%$B: C^\bull(V L, M) \to C^{\bull-1}(V L, M)$ 
denote the 
horizontal differentials of the mixed
 complex associated to the cocyclic
 module from
%s from Theorem
% \ref{bufalinadachteich} and 
Lemma
 \ref{nepaseffacer}.
%, respectively. 
Hence
$\tilde B = N \gs_{-1}(1 - \gl)$, where $\gl :=
(-1)^n \tau_n$, $N := \sum^n_{i=0}
 \gl^i$, and 
$\gs_{-1} := \gs_{n-1}
\tau_n$. Explicitly, we obtain
%$B: B^n(V L, M)\to B^{n-1}(V L, M)$ 
%is given 
%explicitly as
\begin{equation*}
\begin{split}
		  \tilde B(&u_0\otimes_\ahha\cdots\otimes_\ahha u_n \otimes_\ahha m) \\
&=\sum_{i=0}^n\Big((-1)^{ni}\eps(u_0)u_{i+1}\otimes_\ahha\cdots\otimes_\ahha u_n\otimes_\ahha u_1\otimes_\ahha\cdots\otimes_\ahha u_{i-1} \otimes_\ahha m \\
&\hspace*{1.5cm} -(-1)^{n(i-1)}\eps(u_n)u_{i+1}\otimes_\ahha\cdots\otimes_\ahha u_{n-1}\otimes_\ahha u_0\otimes_\ahha\cdots\otimes_\ahha u_{i-1} \otimes_\ahha m \Big).
\end{split}
\end{equation*}
Note that $B^n(V L,M) \cong C^{n+1}(V L,M)$ as $(A,A)$-bimodules in this example.
From our general considerations in
 Section \ref{effacersvp}, we have  
$B \circ \phi \circ \pi =
\phi \circ \pi \circ \tilde B$, where $\pi$ is the canonical projection $B^\bull(U,M) \to U^{\otimes_\ahha \bull+1} \otimes_\Uop M$ and
$\phi: U^{\otimes_\ahha \bull+1} \otimes_\Uop M \to
C^n(V L,M)$ is the isomorphism \rmref{doppellaeufer}. 
Using its right inverse mentioned in the proof of Proposition \ref{excellentesoiree}, 
it is seen that
$$
\Xi(m \otimes_\ahha X_1 \wedge \cdots \wedge X_n) =
\phi \Big( \pi
\big(\textstyle\frac{1}{n!}\sum_{\gs \in S_n} (-1)^\gs 1
\otimes_\ahha X_{\gs(1)} \otimes_\ahha
\cdots \otimes_\ahha  X_{\gs(n)} \otimes_\ahha m \big)\Big).
$$ 
Hence, because $L \subset \ker \eps$ 
%and because $\eps$ is a left $A$-module map, 
we can compute 
\begin{equation*}
\begin{split}
B &\big(\Xi(m \otimes_\ahha X_1 \wedge \cdots \wedge X_n)\big) = \\
%&= B_\emme \phi\pi \Big(\frac{1}{n!}
%\sum_{\gs \in S_n} (-1)^\gs a
%\otimes_\ahha X_{\gs(1)} \otimes_\ahha \cdots \otimes_\ahha X_{\gs(n)}\Big) \\
&= 
\phi \Big( \pi \big(\tilde B (\textstyle\frac{1}{n!}\sum_{\gs \in S_n} (-1)^\gs 1
\otimes_\ahha X_{\gs(1)} \otimes_\ahha \cdots \otimes_\ahha X_{\gs(n)}\otimes_\ahha m ) \big) \Big) \\
&= 
\phi \big(\pi ( \textstyle\frac{1}{(n-1)!} 
\sum_{\gs \in S_n} (-1)^\gs X_{\gs(1)} \otimes_\ahha
\cdots \otimes_\ahha  X_{\gs(n)} \otimes_\ahha m )\big) \\
&= 
{\textstyle\frac{1}{(n-1)!} \sum_{\gs \in S_n}} (-1)^\gs
{X_{\gs(1)}}_- \cdot \big(X_{\gs(2)} \otimes_\ahha
\cdots \otimes_\ahha  X_{\gs(n)} \big)\otimes_\ahha m{X_{\gs(1)}}_+, \\
&= 
\textstyle \frac{1}{(n-1)!} \sum_{\gs \in S_n} (-1)^\gs 
X_{\gs(2)} \otimes_\ahha \cdots \otimes_\ahha  X_{\gs(n)} \otimes_\ahha mX_{\gs(1)}\\
&\quad - \textstyle \frac{1}{(n-1)!} 
\sum^n_{i=1}\sum_{\gs \in S_n} (-1)^\gs X_{\gs(2)} \otimes_\ahha \cdots \otimes_\ahha X_{\gs(1)}X_{\gs(i)} \otimes_\ahha 
\cdots \otimes_\ahha  X_{\gs(n)} \otimes_\ahha m\\
&= \Xi \big(\textstyle \sum^n_{i=1} (-1)^{i + 1} mX_i \otimes_\ahha X_1 \wedge
\cdots \wedge \hat{X}_i \wedge \cdots \wedge X_n  \\
&\qquad\qquad + \textstyle \sum_{i<j} (-1)^{i+j} m \otimes_\ahha [X_i, X_j] \wedge X_1
\wedge \cdots\wedge \hat{X}_i \wedge \cdots \wedge \hat{X}_j \wedge \cdots\wedge X_n
\big) \\
&= \Xi \big(\partial(m \otimes_\ahha X_1 \wedge \cdots \wedge X_n)\big),
\end{split}
\end{equation*}
where $\cdot$ denotes the diagonal action via the coproduct.
This completes the proof.
\end{proof}

\begin{rem}
Note that combining the preceding
 theorem with Proposition \ref{obecni}
 as well as \rmref{chezjeanette} relates
the Hopf-cyclic  
cohomology of $VL$ with the Hopf algebroid homology,
 that is, the simplicial theory of the dual Hopf-cyclic homology:
\begin{align*}
		  HC^\bull(V L, M) 
&\simeq \ker \pl_\bull \oplus H_{\bull-2}(L,M) 
\oplus H_{\bull-4}(L,M) \oplus \cdots\\
&\simeq \ker \pl_\bull \oplus \Tor^{VL}_{\bull-2}(M,A) 
\oplus \Tor^{VL}_{\bull-4}(L,M) \oplus \cdots\\
&\simeq \ker \pl_\bull \oplus H_{\bull-2}(VL,M) 
\oplus H_{\bull-4}(VL,M) \oplus \cdots.
\end{align*}
\end{rem}

\subsection{Twisted cyclic homology}
\label{beispielabsch}
Recall from \cite{Schau:DADOQGHA} that 
$U=\Ae$ is for any $k$-algebra $A$ 
a left Hopf algebroid over 
$A$ with structure maps
$$
		  s(a) := a \otimes_k 1,\quad
		   t(b) := 1 \otimes_k b,\quad 
		  \gD(a \otimes_k b) := (a
		  \otimes_k 1) 
		  \otimes_A (1 \otimes_k b),\quad
		  \eps(a \otimes_k b) := ab.
$$
The inverse of the Hopf-Galois map is given by
$$
		  (a \otimes_k b)_+ \otimes_\Aop 
		  (a \otimes_k b)_- := (a \otimes_k 1)
		  \otimes_\Aop (b \otimes_k 1).
$$

Any algebra endomorphism 
$\gs: A \to A$ defines a right
$\Ae$-module $A_\gs$ which is $A$ as
$k$-module with the right action 
$$
		  x(a \otimes_k b) := 
		  bx\gs(a), \qquad a, x \in A, b \in \Aop.
$$
Define furthermore a left $\Ae$-comodule
structure on $A_\gs$ by 
$$
		  A_\gs \to \Ae \otimes_A A_\gs,
		  \quad 
		  x \mapsto (x \otimes_k 1) \otimes_A 1,
$$
which reduces to the map 
$A_\gs \to \Ae, \ x \mapsto x \otimes_k 1$.
With this $\Ae$-action and
$\Ae$-coaction on $A_\gs$ we have $bx =
xt(b)$, but $xa$ is different from
$xs(a)$ unless $\gs = \mathrm{id}_A$. 
Under the isomorphism 
$C_\bull(\Ae, A_\gs) = 
A_\gs \otimes_\Aopp \Ae^{\otimes_\Aopp
n} 
\simeq A_\gs \otimes_k A^{\otimes_k n}$ 
given by
$$
		  x \otimes_\Aopp (a_1 \otimes_k
		  b_1) \otimes_\Aopp \cdots
		  \otimes_\Aopp (a_n \otimes_k
		  b_n) \mapsto b_n \cdots b_1 x
		  \otimes_k a_1 \otimes_k \cdots
		  \otimes_k a_n,   
$$
the para-cyclic operators 
\rmref{adualnightinpyongyang} become
\begin{equation*}
\!\!\!
\begin{array}{rcll}
d_i(x \otimes_k y)  &\!\!\!\!\! =& \!\!\!\!\!
\left\{ \!\!\!
\begin{array}{l}
a_n x \otimes_k a_1  \otimes_k   \cdots   \otimes_k   a_{n-1}
\\
x \otimes_k \cdots \otimes_k  a_{n-i} a_{n-i+1}
 \otimes_k  \cdots 
\\
x\gs(a_1) \otimes_k a_2  \otimes_k   \cdots    \otimes_k  
a_n 
\end{array}\right.  & \!\!\!\!\!\!\!\!\!\!\!\!   \begin{array}{l} \mbox{if} \ i \!=\! 0, \\ \mbox{if} \ 1
\!  \leq \! i \!\leq\! n-1, \\ \mbox{if} \ i \! = \! n, \end{array} \\
s_i(x \otimes_k y) &\!\!\!\!\! =&\!\!\!\!\!  \left\{ \!\!\!
\begin{array}{l} x  \otimes_k   a_1  
\otimes_k   \cdots   \otimes_k
 a_n  \otimes_k 
		  1
\\
x \otimes_k \cdots \otimes_k   a_{n-i} 
\otimes_k   1  
\otimes_k   a_{n-i+1}  \otimes_k
 \cdots  
\\
x \otimes_k 1
\otimes_k a_1  \otimes_k   \cdots    \otimes_k  a_n 
\end{array}\right.   & \!\!\!\!\!\!\!\!\!\!\!  \begin{array}{l} 
\mbox{if} \ i\!=\!0, \\ 
\mbox{if} \ 1 \!\leq\! i \!\leq\! n-1, \\  \mbox{if} \ i\! = \!n, \end{array} \\
t_n(x \otimes_k y) 
&\!\!\!\!\!=&\!\!\!\!\! 
\gs(a_1) \otimes_k a_2 \otimes_k \cdots
\otimes_k a_n \otimes_k x, 
& \\
\end{array}
\end{equation*}
where we abbreviate $y:=a_1
\otimes_k \cdots \otimes_k a_n$.
In particular, one has
$$
		  t_n^{n+1}=\gs \otimes_k \cdots
		  \otimes_k \gs,
$$
so $C_\bull(\Ae,A_\gs)$ is
cyclic
if and only if $\gs= \id$ (in
which case $A_\gs$ is an SaYD module).

However, there are many situations in
which the canonical projection 
from $C_\bull(\Ae,A_\gs)$ onto its
associated cyclic $k$-module 
$C_\bull(\Ae,A_\gs)/\mathrm{im}(\mathrm{id}-
t_\bull^{\bull+1})$ is a
quasi-isomorphism of the underlying
simplicial $k$-modules, see
e.g.~\cite[Prop.\ 2.1]{HadKra:THOQ}, which implies:
\begin{thm}
If $k$ is a field and $\gs$ is a
diagonalisable automorphism of $A$,
then we have
$$
		  H_\bull(\Ae,A_\gs) \simeq H_\bull(A,A_\gs).
$$
\end{thm}
Here the right hand side denotes the
Hochschild homology of $A$ with
coefficients in the $(A,A)$-bimodule $A_\gs$.
The resulting cyclic homology 
$HC^\gs_\bull(A):=HC_\bull(\Ae,A_\gs)$ has been
first considered in \cite{KusMurTus:DCOQGATCC} under
the name $\gs$-twisted cyclic homology
and has served as yet another 
guiding example of generalised cyclic
homology theories. It can be also
expressed as the
Hopf-cyclic homology of the
$k \mathbb{Z}$-module algebra $A$
(where $k \mathbb{Z}$ acts via
$\gs$), but the above presentation seems
more natural and stresses the way it
originates as a deformation of $HC_\bull(A)$.    
We therefore consider it an important
example that motivates both the
generalisation of Hopf-cyclic
(co)homology from Hopf algebras to Hopf
algebroids, and also the necessity to 
consider
coefficients beyond SaYD modules, and the above shows how to extend the construction of \cite{KusMurTus:DCOQGATCC} to arbitrary $(A,A)$-bimodules assuming the existence of an $\Ae$-coaction on the coefficients.

\providecommand{\bysame}{\leavevmode\hbox to3em{\hrulefill}\thinspace}
\providecommand{\MR}{\relax\ifhmode\unskip\space\fi MR }
% \MRhref is called by the amsart/book/proc definition of \MR.
\providecommand{\MRhref}[2]{%
  \href{http://www.ams.org/mathscinet-getitem?mr=#1}{#2}
}
\providecommand{\href}[2]{#2}

\end{document}